\documentclass{article}
\usepackage[margin=1in]{geometry}
\usepackage[numbers,sort&compress]{natbib}
\usepackage[hidelinks]{hyperref}
\usepackage{amssymb,amsmath}
\usepackage{amsthm}
\usepackage{cleveref}
\usepackage{enumerate}
\usepackage{enumitem}
\usepackage{color}
\usepackage[dvipsnames, table,xcdraw]{xcolor}
\usepackage{fancyvrb}
\usepackage{listings}
\usepackage{graphicx}
\usepackage{booktabs}
\usepackage{subcaption}

\usepackage{mathtools}
\usepackage[parfill]{parskip}
\usepackage{soul}
\usepackage{multirow}
\usepackage{caption}
\usepackage{subcaption}
\usepackage{epsfig,psfrag,url,verbatim}
\usepackage{graphics}
\usepackage{epstopdf}

\usepackage{url}
\usepackage[ruled, linesnumbered]{algorithm2e}
\usepackage{enumitem}
\usepackage{booktabs}
\usepackage{mathtools}
\usepackage{siunitx}
\usepackage{nicefrac}
\usepackage{physics}

\usepackage{float}

\newcommand{\zeros}{\mathbf 0}

\newcommand{\R}{{\mathbb R}}

\newcommand{\N}{{\mathbb N}}

\newcommand{\hoff}[1]{H(#1)}

\newcommand{\prox}{\mathbf{prox}}

\newcommand{\dist}{\mathop{\bf dist{}}}

\newcommand{\proj}{\mathop{\bf proj}}

\newcommand{\argmin}{\mathop{\rm argmin}}

\newcommand{\maximize}{\mathop{\rm maximize}}
\newcommand{\minimize}{\mathop{\rm minimize}}

\newcommand{\bz}{\bar{z}}
\newcommand{\bx}{\bar{x}}
\newcommand{\by}{\bar{y}}

\newcommand{\defeq}{\vcentcolon=}

\newcommand{\lc}{\ell_c}
\newcommand{\uc}{u_c}
\newcommand{\lv}{\ell_v}
\newcommand{\uv}{u_v}
\newcommand{\cR}{\mathcal{R}}
\newcommand{\cY}{\mathcal{Y}}

\newcommand{\cX}{\mathcal{X}}

\newtheorem{theorem}{Theorem}
\newtheorem*{theorem*}{Theorem}
\newtheorem{lemma}{Lemma}

\SetKwFunction{GenericRestartScheme}{GenericRestartScheme}

\SetKwFunction{GenericFixedFrequencyRestartScheme}{GenericFixedFrequencyRestartScheme}

\SetKwFunction{GenericIterativeAlgorithm}{GenericIterativeAlgorithm}

\newcommand{\callGenericIterativeAlgorithm}[1]{\hyperref[function:GenericIterativeAlgorithm]{\GenericIterativeAlgorithm{#1}}}

\SetKwFunction{OneStepOfAlgorithm}{OneStepOfAlgorithm}
\SetKwFunction{InitializeAlgorithm}{InitializeAlgorithm}

\SetKwFunction{OneStepOfPDHG}{OneStepOfPDHG}
\SetKwFunction{InitializePDHG}{InitializePDHG}
\SetKwFunction{AdaptiveRestartPDHG}{AdaptiveRestartPDHG}

\SetKwFunction{OneStepOfAGD}{OneStepOfAGD}
\SetKwFunction{InitializeAGD}{InitializeAGD}
\SetKwFunction{AdaptiveRestartAGD}{AdaptiveRestartAGD}

\SetKwFunction{OneStepOfExtragradient}{OneStepOfExtragradient}
\SetKwFunction{InitializeExtragradient}{InitializeExtragradient}
\SetKwFunction{AdaptiveRestartExtragradient}{AdaptiveRestartExtragradient}

\SetKwFunction{PrimalWeightUpdate}{PrimalWeightUpdate}
\SetKwFunction{AdaptiveStepOfPDHG}{AdaptiveStepOfPDHG}
\SetKwFunction{GetRestartCandidate}{GetRestartCandidate}
\SetKwFunction{InitializePrimalWeight}{InitializePrimalWeight}

\newcommand{\epsZero}{\epsilon_{\text{zero}}}

\SetKwProg{Fn}{Function}{:}{}

\newcommand{\InitialTau}[1]{\tau_{\text{init}}}

\newcommand{\ZStar}{Z^{\star}}

\newcommand{\xStar}{x^{\star}}
\newcommand{\yStar}{y^{\star}}

\newcommand{\Lag}[0]{\mathcal{L}}
\newcommand{\T}{\top}

\newcommand{\ttx}{\tilde{x}}
\newcommand{\tA}{\tilde{A}}

\newcommand{\tc}{\tilde{c}}

\newcommand{\tlc}{\tilde{\lc}}
\newcommand{\tlv}{\tilde{\lv}}
\newcommand{\tuc}{\tilde{\uc}}
\newcommand{\tuv}{\tilde{\uv}}

\newcolumntype{R}{r}
\newcolumntype{C}{c}
\newcolumntype{L}{l}

\usepackage{longtable}

\begin{document}
\lstset{
breaklines=true,
basicstyle=\small\ttfamily,
breakatwhitespace=true,
breakautoindent=false,
breakindent=0pt,
resetmargins=true,
}
	\title{PDLP: A Practical First-Order Method for Large-Scale Linear Programming}

	\author{ David Applegate\thanks{
  Google Research (\texttt{\{dapplegate, wschudy\}@google.com});} \qquad Mateo D\'iaz\thanks{Johns Hopkins University
   (\texttt{mateodd@jhu.edu});} \qquad
  Oliver Hinder\thanks{University of Pittsburgh (\texttt{ohinder@pitt.edu});} \qquad 
  Haihao Lu\thanks{
  MIT and Google Research (\texttt{haihao@mit.edu});} \\
  Miles Lubin\thanks{Hudson River Trading (\texttt{miles.lubin@gmail.com}; work performed primarily while at Google Research)}  \qquad
  Brendan O'Donoghue\thanks{
  Google DeepMind (\texttt{bodonoghue@deepmind.com});}  \qquad
  Warren Schudy$^\ast$
}

\date{}
\maketitle
\begin{abstract}
We present PDLP, a practical first-order method for linear programming (LP) designed to solve large-scale LP problems. PDLP is based on the primal-dual hybrid gradient (PDHG) method applied to the minimax formulation of LP. PDLP incorporates several enhancements to PDHG, including diagonal preconditioning, presolving, adaptive step sizes, adaptive restarting, and feasibility polishing. Our algorithm is implemented in C++, available in Google's open-source OR-Tools library, and supports multithreading.

To evaluate our method, we introduce a new collection of eleven large-scale LP problems with sizes ranging from 125 million to 6.3 billion nonzeros. PDLP solves eight of these instances to optimality gaps of 1\% (with primal and dual feasibility errors of less than $10^{-8}$) within six days on a single machine. We also compare PDLP with Gurobi barrier, primal simplex, and dual simplex implementations. These traditional methods are designed to solve linear programs to much tighter optimality gaps but struggle to solve these instances. Gurobi barrier solves only three instances, exceeding our 1TB RAM limit on the other eight. While primal and dual simplex are more memory-efficient than the barrier method, they are slower and solve only three instances within six days.

Compared with the conference version of this work (in: Advances in Neural Information Processing Systems 34 (NeurIPS 2021)), the key new contributions are: (i) feasibility polishing, a technique that quickly finds solutions that are approximately optimal but almost exactly feasible (without which only two of the eleven problems can be solved); (ii) a multithreaded C++ implementation available in Google OR-Tools; and (iii) a new collection of large-scale LP problems.
Note that the conference version should be referred to for comparisons with SCS and ablation studies, which we do not repeat in this paper.
\end{abstract}

\section{Introduction}

Applications of linear programming (LP) are widespread and span many different fields~\cite{boyd2004convex}. 
Simplex \cite{maros_simplex_encyclopedia} and interior-point methods \cite{wright2005interior} are the standard
approaches to solving LP problems. They 
are widely used in practice and are known for their reliability and
ability to solve instances to high numerical precision.
However, these methods face scaling limitations when applied to extremely large problems.
For interior point methods, the limiting factor is often the memory
required by sparse matrix factorizations.
While the simplex method is less likely to run out of memory because it
factorizes a smaller matrix, it often requires prohibitively many
iterations. Moreover, it is difficult to adapt these methods to exploit highly-parallel
modern computing hardware such as multithreading, distributed computing, and graphics processing units (GPUs).

A parallel stream of research that has had relatively little impact on the practice of LP,
first-order methods (methods that 
use function and gradient information)
have undergone an intense period of study
over the past twenty years 
\cite{duchi2011adaptive,
beck2009fast,
mcmahan2010adaptive,
chambolle2011first,
o2015adaptive,
heyuan2012},
resulting in incredible success
in machine learning \cite{krizhevsky2012imagenet,kaplan2020scaling}.
The benefit of these methods is that 
they have cheap iterations that can readily utilize 
highly parallel modern computing hardware
and quickly obtain approximately optimal solutions, which 
are sufficient for most large-scale machine learning applications.
Building on this literature, this paper enhances an existing first-order method, primal-dual hybrid gradient (PDHG),
popularized by Chambolle and Pock \cite{chambolle2011first} in the computer vision community,
to develop a specialized first-order method for LP.

Despite their appealing computational properties,
a longstanding barrier to adopting first-order methods for LP
is their inability to reliably solve instances to high numerical precision
in reasonable timeframes.
For LP, it is especially critical to find
feasible solutions, both primal and dual.
Even small violations of primal constraints
(for instance, flow conservation) are difficult for
practitioners to interpret. Similarly, finding dual 
feasible solutions is just as essential because of LP's
use for providing bounds on the optimal objective value such as within branch-and-bound methods \cite{wolsey2021integer}.
As a concrete example of the risks associated with an inaccurately solved LP problem, consider the `irish-electricity' instance, an LP instance based on a unit commitment 
case study of the Irish electric grid \cite{miplib_irish_electricity,carroll2017sub}, which has been used in Mittelmann's LP benchmark \cite{mittelmannbenchmark}.
On this instance,
SCS 3.27 \cite{scs}, a state-of-the-art first-order method 
with its default termination tolerance of $10^{-4}$, 
terminates in 40 seconds.
However, the objective value for the returned solution is $2.21 \times 10^6$,
which is off from the true optimal objective value\footnote{Found using Gurobi with crossover enabled.} of 
$2.56 \times 10^6$ by more than 13\%.
On the other hand, SCS with a tighter termination tolerance of $10^{-8}$ does not
terminate successfully within an hour. 
This issue---where moderate feasibility violations lead to such undesirable behavior,
which is only detectable using solutions provided from a more accurate solver---justifies practitioners' caution in adopting first-order methods for LP.

To tackle this reliability issue, 
we apply a series of enhancements to PDHG including: adaptive restarting, dynamic primal-dual step size selection, and diagonal preconditioning. 
We call this enhanced PDHG algorithm, PDLP (PDHG for LP). 
The impact of these enhancements 
was demonstrated in the conference 
version of this paper \cite{applegate2021practical}. 
For example, for finding high accuracy solutions on a standard LP benchmark, NETLIB \cite{netlib},
our enhancements enabled a $75\times$ speed-up over PDHG without any enhancements
and a $19\times$ speed-up \cite[Table~24]{applegate2021practical} over SCS \cite{scs}, a first-order method for conic optimization.

Compared with the conference version of this paper \cite{applegate2021practical}, the new contributions of this paper include:
\begin{enumerate}
\item An implementation of PDLP in multithreaded C++ code which is now part of Google's open-source OR-Tools package (\url{https://developers.google.com/optimization}).
The conference version of our paper uses single-thread Julia code\footnote{The prototype Julia code, FirstOrderLp.jl, is available at \url{https://github.com/google-research/FirstOrderLp.jl}.}.
\item A new collection of large-scale LP instances with between 125 million and 6.3 billion nonzeros to benchmark our method and a set of experiments demonstrating favorable results compared to Gurobi, a state of the art commercial solver.
In comparison, the conference paper focused on
smaller standard benchmark instances: the largest problem tested had 
90 million nonzeros.
\item An additional enhancement that allows our algorithm to rapidly find solutions that are feasible but only approximately optimal.
Our computational results demonstrate this new heuristic is extremely effective for these large-scale problems.
\end{enumerate}

Note that this paper does not subsume the conference version because we do not repeat its numerical experiments. These experiments were:

\begin{enumerate}
	\item An ablation study that individually disabled each enhancement and also compared with reasonable alternatives.
	\item Starting from a naive PDHG implementation, a demonstration of the incremental impact of each  enhancement.
	\item Comparisons with existing solvers (SCS and Gurobi) on instances from medium-size LP benchmarks and a large-scale LP example derived from the PageRank problem \cite{brin1998anatomy}.
\end{enumerate}

Given the C++ code is the exclusive focus of this paper,
we decided neither to repeat results verbatim from the conference paper nor rerun
 these experiments using the C++ code as the conclusions would largely be the same. 
For reference \cite[Table 8]{lu2025cupdlp} shows that the single threaded C++ code achieves roughly a $1.5\times$ speed-up over the Julia code.

\paragraph{Outline}
Section~\ref{sec:pdhg} introduces the PDHG algorithm 
and the notation we use throughout the paper.
Section~\ref{sec:practical-algorithmic-improvements}
explains the practical improvements we have added to 
the PDHG algorithm. Section~\ref{sec:feasibility}
introduces our new feasibility polishing algorithm.
Section~\ref{sec:implementation} describes the C++ implementation. 
Section \ref{sec:instances} describes the collection of large-scale instances.
Section~\ref{sec:numerical-results} presents our numerical results. %

\subsection{Related literature}
\paragraph{Classic methods for linear programming.} 
The simplex method was the first successful algorithm for LP, having been developed by Dantzig and others in the late 1940s, with
improvements over the subsequent decades as the economic importance of LP grew~\cite{maros_simplex_encyclopedia}.
The simplex method begins at a vertex and pivots between vertices, monotonically improving the objective function with each step.
Interior-point methods~\cite{wright2005interior} were a major breakthrough in the 1980s,
coupling much stronger theoretical guarantees with practically successful implementations.
In contrast to the simplex method, interior-point methods (IPMs) traverse the interior of the feasible region, progressing toward optimality by following a so-called central path \cite{renegar1988polynomial}. Commercial and open-source solvers such as Gurobi, CPLEX, COPT, and HiGHS leverage both of these methods, often in combination, to provide reliable and accurate solutions to LP problems.

There have also been several papers that modify interior point methods so that instead of using a factorization to solve the linear systems that arise at each iteration they use preconditioned iterative methods \cite{gondzio2012,fountoulakis2014matrix,zanetti2023new,gondzio2022general,cui2019implementation}. 
This reduces the memory overhead of these methods and the cost of each inner iteration performed by the iterative method (e.g., conjugate gradient) is roughly the cost of a matrix-vector multiplication (assuming the preconditioner is cheap to implement).
For example, \citet{fountoulakis2014matrix} demonstrate excellent performance of these methods on large-scale compressed sensing problems.

\paragraph{PDHG.}
The primal-dual hybrid gradient (PDHG) method, originally proposed in~\cite{zhu2008efficient, esser2010general}, has been extended and analyzed by numerous researchers~\cite{esser2010general, pock2009algorithm, chambolle2011first, he2012convergence, condat2013primal, chambolle2018stochastic, alacaoglu2019convergence}. PDHG is a form of operator splitting~\cite{bauschke2011convex, ryu2016primer}, and it is closely related to the alternating direction method of multipliers (ADMM) and Douglas-Rachford splitting (DRS)~\cite{eckstein1992douglas, boyd2011distributed, o2020equivalence}, all of which can be interpreted as preconditioned instances of the proximal point method (PPM)~\cite{rockafellar1976monotone, eckstein1992douglas, parikh2014proximal,lu2023unified}. 
However, unlike ADMM, DRS, or PPM, when applied to LP problems, PDHG is matrix-free, relying solely on matrix-vector multiplications.

\paragraph{PDLP.} The conference version of this work \cite{applegate2021practical} introduced PDLP, a restarted PDHG with
several practical enhancements specialized for solving LP problems. Experiments were based on a prototype Julia implementation.
This paper includes an additional enhancement, and the experiments discussed in this paper are based on a C++ implementation of PDLP, which has been open-sourced through Google OR-Tools since 2022 \cite{ortools}. More recently, a GPU implementation, dubbed cuPDLP, demonstrated strong numerical performance comparable to commercial LP solvers, which garnered attention from both the academic community and the solver industry \cite{lu2025cupdlp,lu2023cupdlpc}. Variants of (cu)PDLP have been incorporated into both commercial and open-source optimization solvers, including COPT \cite{copt}, FICO Xpress \cite{xpress2014fico}, HiGHS \cite{huangfu2018parallelizing}, and NVIDIA cuOpt \cite{cuopt}.

The success of PDLP has spurred theoretical analysis. Lu and Yang \cite{lu2022infimal} proved that PDHG exhibits a linear convergence rate when applied to LP. Applegate et al. \cite{applegate2023faster} showed that a faster linear convergence rate can be achieved through restarts, with the rate matching the worst-case lower bound. Applegate et al.  \cite{applegate2024infeasibility} demonstrated that for infeasible or unbounded LP instances, PDHG iterates diverge toward a certificate of infeasibility or unboundedness, allowing for infeasibility detection without additional computation. Lu and Yang \cite{lu2024geometry} established that PDHG displays a two-stage behavior when applied to LP: in the first stage, it identifies the optimal active variables in finitely many iterations; in the second stage, it solves a homogeneous linear inequality system at an improved linear rate. This two-stage behavior helps explain the slow convergence observed when an LP instance is near-degenerate. Hinder \cite{hinder2024worst} shows an  iteration complexity bound for restarted PDHG applied to totally unimodular LP 
of $O(H m^{2.5} \sqrt{\text{nnz}(A)} \ln( H n / \epsilon))$ where $H$ is the largest absolute value of the right-hand side and objective vector, $n$ is the number of columns, $m$ is the number of rows, $\text{nnz}(A)$ the number of nonzeros in the constraint matrix, and $\epsilon$ is the distance to optimality measured in Euclidean norm. 
Unlike other results, this bound does not involve any condition numbers.
Xiong and Freund \cite{xiong2024role} introduced the concept of level-set geometry, showing that PDHG converges more quickly for LP instances with regular, non-flat level sets around the optimal solution. Recently, additional numerical improvements have been proposed to further accelerate the algorithm. In particular, Xiong and Freund \cite{xiong2024role} introduced a central-path-based preconditioner for PDHG that significantly improves convergence. Further,
Lu and Yang \cite{lu2024restarted} incorporated the Halpern iteration \cite{halpern1967fixed} into PDLP, achieving accelerated theoretical guarantees and improved practical performance.

\paragraph{Other first-order methods for LP.}

Lan, Lu, and Monteiro \cite{Lan2011} and Renegar \cite{Renegar2019} developed first-order methods (FOMs) for LP as a particular case of semidefinite programming with sublinear convergence rates. The FOM-based solvers apply to more general problem classes like cone or quadratic programming. In contrast, some of the enhancements that constitute PDLP are specialized, either in theory or practice, for LP (namely restarts~\cite{applegate2023faster} and presolving). A number of authors \cite{wang2017admmlp, eckstein1990alternating, gilpin2012first, yang2018rsg, necoara2019linearfom} have proposed linearly convergent FOMs for LP; however, to our knowledge, none have been the subject of a comprehensive computational study. FOM-based software packages that can be used to solve LP include the following:
\begin{itemize}
  \item %
		{SCS} \cite{o2016conic, o2021operator} is an ADMM-based solver for convex quadratic cone programs implemented in C. SCS applies ADMM \cite{boyd2011distributed} to the homogeneous self-dual embedding of the problem \cite{ye1994nl}, which yields infeasibility certificates when appropriate. One of the key computational challenges of ADMM-based methods is solving a linear system at every iteration. SCS provides two options to tackle this subroutine: a direct solver, which computes and reuses the factorization of the matrix, and an indirect solver, which uses conjugate gradient methods to iteratively solve these linear systems, offering scalability trade-offs.
	\item OSQP \cite{stellato2020osqp} is another ADMM-based solver for convex quadratic programming implemented in C. Compared to SCS, a major difference is that OSQP applies ADMM to the original problem rather than the homogeneous self-dual embedding. It also has direct and indirect options for solving the linear systems at each iteration. A related solver, COSMO \cite{garstka_2019}, uses the same algorithm for convex quadratic cone programs and is implemented in Julia.
	\item ECLIPSE \cite{basu2020eclipse} is a distributed LP solver that employs accelerated gradient descent to solve a smoothed dual formulation of LPs. It is specifically designed to handle large-scale LPs with particular decomposition structures, such as those encountered in web applications.
  \item %
		{ABIP and ABIP+} \cite{lin2021admm} are interior-point method (IPM) solvers for cone programming. ABIP uses a few steps of ADMM to approximately solve the sub-problems that arise when applying a path-following barrier algorithm to the homogeneous self-dual embedding of the problem. Although IPMs are technically second-order methods, the use of ADMM in the inner loop allows the solver to scale to larger problems than traditional IPMs can handle.  ABIP+ \cite{deng2022new} adapts many of the improvements from PDLP to ABIP to improve its performance.
	\item TFOCS \cite{becker2011templates} is a MATLAB solver designed for solving a variety of convex cone problems that arise in areas such as signal processing, machine learning, and statistics, including applications like the Dantzig selector, basis pursuit denoising, and image recovery. In principle, TFOCS can also be applied to LP. It utilizes a variant of Nesterov's accelerated method~\cite{nesterov1983method} to solve a smoothed version of the primal-dual conic problem.
\item ALGENCAN \cite{andreani2008augmented} is a widely established software package for nonlinear programming that is based on the augmented Lagrangian method. The augmented Lagrangian subproblems are  solved by the inner solver GENCAN \cite{birgin2002large}, which utilizes active-set strategies and spectral projected gradients. Crucially, ALGENCAN is designed to operate in a fully matrix-free manner, i.e., only using gradient evaluations and Hessian-vector products.
\end{itemize}

\section{Preliminaries}\label{sec:pdhg}

In this section, we introduce the notation we use throughout the paper, summarize the LP formulations we solve, and introduce the baseline PDHG algorithm.

\paragraph{Notation.}
Let $\R$ denote the set of real numbers, $\R^{+}$ the set of nonnegative real numbers, and $\R^{-}$ the set of nonpositive real numbers. Let $\N$ denote the set of natural numbers (starting from one). We endow $\R^n$ with the standard dot product $\langle x, y \rangle = x^\top y$ and its induced norm $\| \cdot \|_2.$ We will also use $\| \cdot \|_p$ to denote the $\ell_p$ norm for a vector. The $\ell_{\infty}$-norm, $\|v \|_\infty$, computes the maximum absolute value of the vector $v.$ For a vector $v \in \R^n$, we use $v^+$ and $v^-$ for their positive and negative parts, i.e., $v_i^+ = \max\{0, v_i\}$ and $v_i^- = -\min\{0, v_i\}$. 
The symbol $\mathbf{1}$ denotes the vector of all ones. Given a matrix $A \in \R^{m \times n},$ the symbols $A_{i,\cdot}$ and $A_{\cdot,j}$ correspond to the $i$th row and $j$th column of $A$, respectively. We use $\|\cdot\|_2$ to denote the spectral norm for a matrix, i.e., its top singular value. 
We use $z^{k} = (x^{k}, y^{k})$ to denote the $k$th iterate of PDHG and
 $z^{n,t} = (x^{n,t}, y^{n,t})$ to denote iterates of PDLP where $n$ is the outer iteration counter and $t$ is the inner iteration counter.

Let $\cX \subseteq \R^n$ be a closed convex set. The orthogonal projection onto $\cX$ is given by $$\proj_{\cX}(x) = \argmin_{w\in \cX} \|x- w\|_2.$$ Define the mapping $p \colon \R^n \times (\R \cup \{-\infty\})^n \times (\R \cup \{\infty\})^n \to \R \cup \{\infty\}$ given by
$$
p(y; \ell, u) := u^\T y^+ -  \ell^\T y^- .
$$
The function $p(\cdot ; \ell, u)$ is piecewise linear and convex for any two vectors $\ell \in  (\R \cup \{-\infty\})^n$ and $u \in (\R \cup \{\infty\})^n$ satisfying $\ell \leq u$. With a slight abuse of notation, we will also use $p$ to denote an analogous function taking $m$-dimensional inputs as implied by the context.

\paragraph{Linear Programming.} We solve primal-dual LP problems of the form \cite{Vanderbei2020}:
\begin{equation}\label{eq:lp}
    \begin{aligned}[c]
    \minimize_{x\in \R^n}~~ &~ c^\top x & \\
\text{subject to:}~~ &~ \lc \leq Ax \leq \uc \\
& ~ \lv \leq x \leq  \uv,
    \end{aligned}
    \qquad\qquad\qquad
    \begin{aligned}[c]
\maximize_{y\in \R^{m}, r \in\R^{n}} \quad &  - p(-y; \lc, \uc) - p(-r ; \lv, \uv) \\
\text{subject to:}\quad & c - A^\T y = r \\
&  y \in \cY \\
& r \in \cR.
    \end{aligned}
\end{equation} 
where $A \in \R^{m \times n}$, $c \in \R^{n}$, $\lc \in (\R \cup \{- \infty\})^m$, $\uc \in (\R \cup \{\infty\})^m$, $\lv \in (\R \cup \{ -\infty \})^{n}$, $\uv \in (\R \cup \{ \infty \})^{n}$, and the sets $\cY\subseteq \R^{m}$ and $\cR \subseteq \R^{n}$ are Cartesian products with their $i$th components given by 
\begin{equation*}
\begin{aligned}
    &\cY_i := \begin{cases}
\{ 0 \} & (\lc)_i = -\infty, ~ (\uc)_i = \infty,  \\
\R^{-} & (\lc)_i = -\infty, ~ (\uc)_i \in \R ,\\
\R^{+} & (\lc)_i \in \R, ~ (\uc)_i = \infty, \\
\R & \text{otherwise}; 
\end{cases}
\quad \text{ and } \qquad
\end{aligned}
\begin{aligned}
    &\cR_i := \begin{cases}
\{ 0 \} & (\lv)_i = -\infty, ~ (\uv)_i = \infty,  \\
\R^{-} & (\lv)_i = -\infty, ~ (\uv)_i \in \R, \\
\R^{+} & (\lv)_i \in \R, ~ (\uv)_i = \infty, \\
\R & \text{otherwise};
\end{cases}
\end{aligned}
\end{equation*}
which are the values of $y$ and $r$ with a finite dual objective value.
We highlight that the formulation we consider \eqref{eq:lp} is slightly different from the formulation in the conference version of this paper~\cite{applegate2021practical}, which only considers a lower bound $\lc \leq Ax$ and has additional equality constraints $Gx = q$. 

This pair of primal-dual problems is equivalent to the minimax problem \cite{Vanderbei2020}: %
\begin{flalign}\label{eq:primal-dual}
\min_{x \in \cX} \max_{y \in \cY} \Lag(x,y) := c^\T x -    y^\T A x - p(y; -\uc, -\lc)
\end{flalign}
with $\cX := \{x \in \R^n : \lv \leq x \leq \uv \}$. Notice that the variable $r$ in \eqref{eq:lp} does not appear explicitly in the minimax formulation. 
This absence leaves an algorithmic choice for recovering the reduced costs, $r$, which impacts when the termination criteria is satisfied
but does not affect the iterates themselves.
A natural choice is to set $r = \proj_{\mathcal{R}}(c - A^\T y)$; we discuss a more advanced choice in Appendix~\ref{sec:reduced-cost-selection}.

\paragraph{PDHG.} When specialized to \eqref{eq:primal-dual}, the PDHG algorithm takes the form:
\begin{subequations}\label{eq:specialized-pdhg}
\begin{flalign}
x^{k+1} &= \proj_{\cX}(x^k - \tau (c - A^\top y^k)) \label{eq:specialized-pdhg:primal-step}\\
  y^{k+1} &= %
            y^k - \sigma A \left(2 x^{k+1} - x^k\right) - \sigma \proj_{[-\uc, -\lc]}\left(\sigma^{-1} y^k - A \left(2 x^{k+1} - x^k\right)\right) \label{eq:specialized-pdhg:dual-step}
\end{flalign}
\end{subequations}
where $\tau, \sigma > 0$ are primal and dual step sizes, respectively. 
For completeness we provide the derivation of \Cref{eq:specialized-pdhg} from PDHG in Appendix~\ref{sec:derive-pdhg-for-this-paper}.
A quick case-by-case argument shows \(y^{k+1} \in \cY\).
For instance, suppose the $i$-th coordinate bounds satisfy \((\lc)_{i} = -\infty\) and \((\uc)_{i} \in \mathbb{R}\), and let
$
\widehat{y}^{k} \;=\; y^k \;-\; \sigma\,A\left(2\,x^{k+1} \;-\; x^k\right).
$
Then, we have
\[
y_{i}^{k+1} \;=\;
  \begin{cases}
    0 & \text{if } \widehat{y}_{i}^{k} \;\ge\; -\sigma\,(\uc)_{i}, \\
    \widehat{y}_{i}^{k} \;+\; \sigma\,(\uc)_{i} & \text{otherwise},
  \end{cases}
\]
and, thus, in any case \(y_{i}^{k+1} \in \R^{-} = \cY_{i}\), as we wanted.
The other cases follow analogously.

PDHG is known to converge to an optimal solution when $\tau \sigma \| A \|_2^2 < 1$ \cite{condat2013primal,chambolle2016ergodic}.
We reparameterize the step sizes by \begin{equation}
    \tau = \eta / \omega  \quad \text{and} \quad \sigma = \omega \eta \qquad \text{with }\eta \in (0, \infty) \quad \text{and} \quad \omega \in (0,\infty).
\end{equation}
We call $\omega \in (0, \infty)$ the \emph{primal weight}, and $\eta \in (0,\infty)$ the \emph{step size}. Under this reparameterization, PDHG converges for all $\eta < 1 / \| A \|_2$. This parameterization allows us to control the scaling between the primal and dual iterates with a single parameter $\omega$. From now on and for the rest of the paper, we will use $z$ as a placeholder for the concatenation of primal and dual vectors $(x,y).$ We use the term \emph{primal weight} to describe $\omega$ because it is the weight on the primal variables in the following norm:
$$
\| z \|_{\omega} := \sqrt{ \omega \| x \|_2^2 + \frac{\| y \|_2^2}{\omega} }.
$$
This norm plays a role in the theory for PDHG \cite{chambolle2016ergodic} and later algorithmic discussions.

\section{Practical algorithmic improvements}\label{sec:practical-algorithmic-improvements}

\newcommand{\zc}[0]{z_{\text{c}}}
\newcommand{\rc}[0]{r_{\text{c}}}

In this section, we briefly describe algorithmic enhancements for PDHG that were introduced in the conference version of this paper~\cite{applegate2021practical}. We omit a discussion of presolve because we do not use it for the numerical experiments in this paper; see also the mention of presolve in Section~\ref{sec:implementation}. A separate section (Section~\ref{sec:feasibility}) is reserved to introduce feasibility polishing, which is new to this work.

Note that while our enhancements are inspired by theory, our focus is on practical performance. The algorithm as a whole does not have convergence guarantees. Guarantees for certain individual enhancements are noted at the end of the section. %

Algorithm~\ref{alg:practical-algorithm} presents pseudo-code for PDLP. %
We modify the step sizes, %
add restarts,  
dynamically update the primal weights, %
and use multithreading. %
Before running Algorithm~\ref{alg:practical-algorithm} we apply %
diagonal preconditioning. %

\paragraph{Periodic checks of restarts and termination criteria} There are some minor differences between the pseudo-code and the actual code.
In particular, we only evaluate the restart or termination criteria (Line~\ref{line:restart-or-termination}) every $64$ iterations\footnote{Condition (iii) of the restart criteria is checked every iteration.}. This reduces the associated overhead of these evaluations with minimal impact on the total number of iterations. We also
check the termination criteria before beginning the algorithm and if we detect a numerical error.

\begin{algorithm}
\SetAlgoLined
{\bf Input:} An initial solution $z^{0,0}$\;
Initialize outer loop counter $n\gets 0$, total iterations $k \gets 0$, step size $\hat \eta^{0,0} \gets 1/\| A \|_{\infty}$,  primal weight $\omega^{0} \gets$ \InitializePrimalWeight{$c, \lc, \uc$}\;
 \Repeat{termination criteria holds}{
 $t \gets 0$\;
 \Repeat{restart or termination criteria holds}{
    $z^{n,t+1}, \eta^{n,t+1}, \hat \eta^{n,t+1} \gets$ \AdaptiveStepOfPDHG{$z^{n,t}, \omega^{n}, \hat \eta^{n,t}, k$} \;
$\bz^{n,t+1}\gets\frac{1}{\sum_{i=1}^{t+1} \eta^{n,i}} \sum_{i=1}^{t+1} \eta^{n,i} z^{n,i}$ \; 
    $\zc^{n,t+1} \gets$ \GetRestartCandidate{$z^{n,t+1}, \bz^{n,t+1}, z^{n,0}$} \;
    $t\gets t+1$, $k \gets k + 1$ \;
 }\label{line:restart-or-termination}
  \textbf{restart the outer loop.} $z^{n+1,0}\gets \zc^{n,t}$, $\hat \eta^{n+1,0} \gets \hat \eta^{n,t}$, $n\gets n+1$ \;
 \label{alg:primal-weight-updates} $\omega^{n} \gets $\PrimalWeightUpdate{$z^{n,0}, z^{n-1,0}, \omega^{n-1}$} \;
 }
 {\bf Output:} $z^{n,0}$.
 \caption{PDLP (after preconditioning and presolve)}
 \label{alg:practical-algorithm}
\end{algorithm}

\begin{algorithm}
\SetAlgoLined
\Fn{\AdaptiveStepOfPDHG($z^{n, t}$, $\omega^n$, $\hat \eta^{n,t}, k$)}{
$(x, y) \gets z^{n, t}$, $\eta \gets \hat \eta^{n,t}$ \;
 \For{$i=1,\dots,\infty$}{
   $x' \gets \proj_{\mathcal{X}}(x - \frac{\eta}{\omega^{n}} (c - A^\top y))$ \;
   $y' \gets y -  \eta \omega^{n} A(2x' - x) - \eta \omega^{n}\proj_{[-\uc, -\lc]}\left( (\eta \omega^{n})^{-1}y -  A(2x' - x)  \right) $\;
   $\bar \eta \gets \begin{cases}
    \frac{\| (x'-x, y'-y) \|_{\omega^{n}}^2 }{2(y'-y)^\T A(x'-x)} & 2(y'-y)^\T A(x'-x) > 0  \\
    \infty & \text{otherwise}
   \end{cases}$ \; 
   $\eta' \gets \min\left(
      (1 - (k + 1)^{-0.3}) \bar \eta,
      ( 1 + (k + 1 )^{-0.6}) \eta\right)$ \;
   \If{$\eta \le \bar \eta$}{
     \Return{$(x', y')$, $\eta$, $\eta'$} 
   }
   $\eta \gets \eta'$ \;
 }
}
 \caption{One step of PDHG using our step size heuristic}
 \label{alg:step-size}
\end{algorithm}

\paragraph{Step size choice.}\label{sec:step-size-choice-body}
The convergence analysis \cite[Equation (15)]{chambolle2016ergodic}
of standard PDHG~\eqref{eq:specialized-pdhg}
 relies on a small constant step size
\begin{equation}
      \eta \le \frac{\| z^{k+1}-z^{k} \|_\omega^2}{2(y^{k+1}-y^k)^\T A(x^{k+1}-x^{k})} \label{eq:acceptableStep}
\end{equation}
where $z^{k} = (x^{k}, y^{k})$.
Classically one would ensure \eqref{eq:acceptableStep} by picking $\eta = \frac{1}{\|A\|_2}$. This is overly pessimistic and requires estimation of $\|A\|_2$. Instead our \AdaptiveStepOfPDHG adjusts $\eta$ dynamically to ensure that \eqref{eq:acceptableStep} is satisfied. If \eqref{eq:acceptableStep} isn't satisfied, we abort the step; i.e., we  reduce $\eta$, and try again. If \eqref{eq:acceptableStep} is satisfied we accept the step. This is described in Algorithm~\ref{alg:step-size}.
Note that in Algorithm~\ref{alg:step-size} $\bar \eta \ge \frac{1}{\|A\|_2}$ holds always, and from this one can show the step size $\eta$ satisfies $\eta \ge \frac{1 - o(1)}{\|A\|_2}$ as $k$ tends to infinity.

This step size is compared against other standard step size schemes in Appendix~C.1 of \citet{applegate2021practical}.

\paragraph{Adaptive restarts.}\label{sec:adaptive-restarts}
In PDLP, we adaptively restart the PDHG algorithm at each outer iteration.
To decide when to restart at the $n$th outer iteration we use the normalized duality gap at $z$, which for any radius $r \in (0,\infty)$, is defined by
$$
\rho_{r}^n(z) := \frac{1}{r} \max_{(\hat{x},\hat{y}) \in \{  \hat{z} \in Z : \| \hat{z} - z \|_{\omega^{n}} \le r \}}\{ \Lag(x, \hat{y}) - \Lag(\hat{x}, y)\}
$$
where $Z = \cX \times \cY$ and for completeness define $\rho_0(z) := \limsup_{r \rightarrow 0^{+}} \rho_{r}^n(z)$.
This measure was introduced by \citet{applegate2023faster}.  Unlike the standard duality gap,
$
\max_{(\hat{x},\hat{y}) \in Z} \{ \Lag(x, \hat{y}) - \Lag(\hat{x}, y)\},
$
the normalized duality gap is always a finite quantity. Furthermore, for any value of $r$ and $\omega^n$, the normalized duality gap $\rho_r^n(z)$ is $0$ if and only if the solution $z$ is an optimal solution to \eqref{eq:primal-dual} \cite{applegate2023faster}; thus, it provides a metric for measuring progress towards the optimal solution. We compute the normalized duality gap based on the linear time implementation given by \citet[Section 3.2]{applegate2023faster}. For brevity, define $\mu_n(z, z_{\text{ref}})$ as the normalized duality gap at $z$ with radius $\| z -  z_{\text{ref}} \|_{\omega^{n}}$, i.e.,
$$
\mu_n(z, z_{\text{ref}}) := \rho_{\| z -  z_{\text{ref}} \|_{\omega^{n}}}^n(z) ,
$$
where $z_{\text{ref}}$ is a user-chosen reference point.

\emph{Choosing the restart candidate.}
To choose the restart candidate $\zc^{n,t+1}$ we call %
$$
\text{\GetRestartCandidate{$z^{n,t+1}, \bz^{n,t+1}, z^{n,0}$}} := \begin{cases}
z^{n,t+1} & \mu_{n}(z^{n,t+1}, z^{n,0}) < \mu_n(\bz^{n,t+1}, z^{n,0}) \\
\bz^{n,t+1} & \text{otherwise}  \ .
\end{cases}
$$
This choice is justified in Remark~5 of \cite{applegate2023faster}.

\emph{Restart criteria.} We define three parameters: $\beta_{\text{necessary}} \in (0,1)$, $\beta_{\text{sufficient}} \in (0, \beta_{\text{necessary}})$ and $\beta_{\text{artificial}} \in (0,1)$. In PDLP we use $\beta_{\text{necessary}} = 0.9$, $\beta_{\text{sufficient}} = 0.1$, and $\beta_{\text{artificial}} = 0.5$. The algorithm restarts if one of three conditions holds:

\begin{itemize}[leftmargin=.7cm]
    \item[(i)] (\textbf{Sufficient decay in normalized duality gap})
$
\mu_n(\zc^{n,t+1}, z^{n,0}) \le \beta_{\text{sufficient}} \mu_n(z^{n,0}, z^{n-1,0}) \ ,
$
    \item[(ii)] (\textbf{Necessary decay + no local progress in normalized duality gap}) 
    \begin{equation*}
        \mu_n(\zc^{n,t+1}, z^{n,0}) \le \beta_{\text{necessary}} \mu_n(z^{n,0}, z^{n-1,0}) \quad \text{and} \quad \mu_n(\zc^{n,t+1}, z^{n,0}) > \mu_n(z^{n,t}_c, z^{n,0}) \ ,
    \end{equation*}
        \item[(iii)]  (\textbf{Long inner loop})
    $t \ge \beta_{\text{artificial}} k\ .$
\end{itemize}
The motivation for (i) is presented in \cite{applegate2023faster}; it guarantees the linear convergence of restarted PDHG on LP problems. 
The second condition in (ii) is inspired by adaptive restart schemes for accelerated gradient descent where restarts are triggered if the function value increases \cite{o2015adaptive}.
The first inequality in (ii) provides a safeguard for the second one, preventing the algorithm from restarting every inner iteration or never restarting.
The motivation for (iii) relates to the primal weight updates. In particular, primal weight updates only occur after a restart (see Line~\ref{alg:primal-weight-updates} of Algorithm~\ref{alg:practical-algorithm}), and condition (iii) ensures that the primal weight will be updated infinitely often.
This prevents a bad choice of primal weight in earlier iterations from causing progress to stall for a long time.

The practical improvement from our restart scheme is demonstrated in Appendix~C.2 of \citet{applegate2021practical}.

\paragraph{Primal weight updates.}
Algorithm~\ref{alg:primal-weight-update} aims to choose the primal weight $\omega^{n}$ such that distance to optimality from primal and dual iterates is the same, i.e.,
$\| (x^{n,t} - \xStar, \mathbf{0}) \|_{\omega^n} \approx \| (\mathbf{0}, y^{n,t} - \yStar) \|_{\omega^n}$. By definition of $\| \cdot \|_{\omega}$,
$$
\| (x^{n,t} - \xStar, \mathbf{0}) \|_{\omega^n} = \sqrt{\omega^n} \| x^{n,t} - \xStar \|_2, \quad \| (\mathbf{0}, y^{n,t} - \yStar) \|_{\omega^n} = \frac{1}{\sqrt{\omega^n}} \| y^{n,t} - \yStar \|_2.
$$
Setting these two terms equal yields
$\omega^n = \frac{\| y^{n,t} - \yStar \|_2}{\| x^{n,t} - \xStar \|_2}$.
Evidently, the quantity $\frac{\| y^{n,t} - \yStar \|_2}{\| x^{n,t} - \xStar \|_2}$ is unknown beforehand, but we attempt to estimate it using $\Delta^n_y / \Delta^n_x$. However, the quantity $\Delta^n_y / \Delta^n_x$ can change wildly from one restart to another, causing $\omega^n$ to oscillate. To dampen variations in $\omega^n$, we first move to a log-scale, in which the primal weight is symmetric, i.e., $\log(1 / \omega^n) = -\log(\omega^n)$, and perform an exponential smoothing with parameter $\theta \in [0,1]$. In PDLP, we use $\theta = 0.5$. 
The parameter $\epsZero = 10^{-10}$ is a small nonzero tolerance that, in principle, could be adjusted.
It helps prevent numerical issues or other performance issues occurring from very huge or 
very small movement of the iterates resulting in a very dramatic change in the primal weights.

\begin{algorithm}
\caption{Primal weight update}\label{alg:primal-weight-update}
\SetAlgoLined
\Fn{\PrimalWeightUpdate{$z^{n,0}, z^{n-1,0}, \omega^{n-1}$}}{
$\Delta_x^n = \| x^{n,0} - x^{n-1,0} \|_2, \quad \Delta_y^n = \| y^{n,0} - y^{n-1,0} \|_2$ \;
\uIf{$\Delta_x^n \in (\epsZero,\epsZero^{-1})$ and $\Delta_y^n \in (\epsZero,\epsZero^{-1})$}{
\Return{ $\exp\left(\theta \log\left( \frac{\Delta_y^n}{\Delta_x^n} \right) + (1 - \theta) \log(\omega^{n-1}) \right)$}
}
\Else{
\Return{$\omega^{n-1}$} \;
}
}
\end{algorithm}

\paragraph{Initialization of the primal weights.}
Define the map \texttt{Combine}$(\lc, \uc)$ which outputs a vector in $\mathbb{R}^{m}$ whose $i$th coordinate is given by
$\max\left\{0, (\lc)_{i}^\star, (\uc)_{i}^\star\right\}$
where for any $x \in \mathbb{R} \cup \{\pm \infty\}$, the operation $(x)^{\star}$ is equal to $|x|$ if it is finite or equal to zero otherwise.
The primal weight is initialized using
\begin{flalign*}
  \InitializePrimalWeight(c, \lc, \uc) :=
  \begin{cases} \frac{\| c \|_2}{\| \texttt{Combine}(\lc, \uc) \|_2} & \| c \|_2 > 0.0 \text{ and } \| \texttt{Combine}(\lc, \uc) \|_2 > 0.0 \\
1 & \text{otherwise}.
\end{cases}
\end{flalign*}

There are several important differences between our primal weight heuristic and proposals from Goldstein et al.~\cite{goldstein2013adaptive,goldstein2015adaptive}. For example, Goldstein et al.~make relatively small changes to the primal weights at each iteration, attempting to balance the primal and dual residual. These changes have to be diminishingly small because, in our experience, PDHG may be unstable if they are too big.  In contrast, in our method, the primal weight is only updated during restarts, which in practice allows for much larger changes without instability issues. Moreover, our scheme tries to balance the weighted distance traveled by the primal and dual iterates rather than the residuals.

The practical improvement from our primal weight scheme is demonstrated in Appendix~C.3 of \citet{applegate2021practical}.

\paragraph{Diagonal Preconditioning.}\label{sec:diag}
Preconditioning is a popular heuristic in  optimization for improving the convergence of first-order methods.
To avoid factorizations, we only consider diagonal preconditioners. Our goal is to rescale the constraint matrix $A$ to $\tA=D_1 A D_2$ with positive diagonal matrices $D_1$ and $D_2$, so that the resulting matrix $\tA$ is ``well balanced''. Such preconditioning creates a new LP instance that replaces $A, c, (\lc, \uc),$ and $ (\lv, \uv)$ in \eqref{eq:lp} with
$\tA$, $\ttx = D_2^{-1} x$, $\tc=D_2 c$, $(\tlc,\tuc)= D_1 (\lc,\uc)$, and $(\tlv, \tuv) = D_2^{-1} (\lv, \uv).$
Common choices for $D_1$ and $D_2$ include:
\begin{itemize}[leftmargin=*]
    \item \textbf{Pock-Chambolle \cite{pock2011diagonal}:} Pock and Chambolle proposed a family of diagonal preconditioners\footnote{Diagonal preconditioning is equivalent to changing to a weighted $\ell_2$ norm in the proximal step of PDHG (weight defined by $D_2$ and $D_1$ for the primal and dual respectively). Pock and Chambolle use this weighted norm perspective.} for PDHG parameterized by $\alpha$, where the diagonal matrices are defined by $(D_1)_{jj}=1/\sqrt{\|A_{j,\cdot}\|_{2-\alpha}}$ for $j=1,...,m$ and $(D_2)_{ii}=1/\sqrt{\|A_{\cdot,i}\|_{\alpha}}$ for $i=1,...,n$. We use $\alpha=1$ in PDLP. This is the baseline diagonal preconditioner in the PDHG literature.
  \item \textbf{Ruiz~\cite{ruiz2001scaling}:} %
        In an iteration of Ruiz scaling, the diagonal matrices are defined as $(D_1)_{jj}=1/\sqrt{\|A_{j,\cdot}\|_{\infty}}$ for $j=1,...,m$ and $(D_2)_{ii}=1/\sqrt{\|A_{\cdot,i}\|_{\infty}}$ for $i=1,...,n$. Ruiz \cite{ruiz2001scaling} shows that if this rescaling is applied iteratively, the infinity norm of each row and each column converge to $1$.
\end{itemize}
For the default PDLP settings, we apply a combination of Ruiz rescaling~\cite{ruiz2001scaling} and the preconditioning technique proposed by Pock and Chambolle \cite{pock2011diagonal}. In particular, we apply $10$ iterations of Ruiz scaling and then apply the Pock-Chambolle scaling. 
Appendix C.5 of \citet{applegate2021practical} shows that PDLP's diagonal preconditioning outperforms either Pock-Chambolle or Ruiz individually.

\paragraph{Infeasibility detection.} Inspired by \cite{applegate2024infeasibility}, PDLP periodically checks whether any of three sequences---the difference of consecutive iterates, the normalized iterates, or the normalized running average---can serve as approximate certificates of infeasibility. We evaluate running averages starting from the most recent restart.
This allows PDLP to detect infeasibility with little additional overhead.
 Notably, verifying all three sequences is often faster than relying on just one, since their convergence speed varies with the problem’s geometry \cite{applegate2024infeasibility}.

\paragraph{Theoretical guarantees for the above enhancements.}
\label{sec:guarantees}
While PDLP's enhancements are motivated by theory, some of them may not preserve theoretical guarantees: we do not have a proof of convergence for the adaptive step size rule or primal weight update scheme.
However, Applegate et al. \cite{applegate2023faster} showed in a simpler setting with fixed primal weight and step size that adaptive restarts preserve convergence guarantees. Similarly, the theoretical guarantees for infeasibility detection \cite{applegate2024infeasibility} do not account for restarts.

\section{Quickly finding approximately optimal but feasible solutions}
\label{sec:feasibility}

In integer programming, approximately optimal but feasible solutions are typically acceptable. 
For example, branch-and-bound methods are commonly stopped as soon as the optimality gap falls below a predefined 
threshold (e.g., 1\%) to avoid excessive computation time while still delivering high-quality solutions \cite{gurobi_mipgap}.
This pragmatic termination criterion motivates 
a new heuristic we develop, called feasibility polishing, that enables first-order methods 
to similarly find solutions with extremely small feasibility violations and moderate (e.g., 1\%) duality gaps.

There are two key insights behind our heuristic:
\begin{enumerate}
\item \textbf{Restarted PDHG tends to converge faster for feasibility problems than for optimality problems.} By a feasibility problem, we mean a problem with no objective, i.e.,
\begin{equation}\label{primal-lp-feasibility}
    \begin{aligned}[c]
    \minimize_{x\in \R^n}~~ &~ 0 & \\
\text{subject to:}~~ &~ \lc \leq Ax \leq \uc \\
& ~ \lv \leq x \leq  \uv.
    \end{aligned}
\end{equation}
Empirically, these problems can be orders of magnitude faster for PDHG to solve compared to the original linear program (with nonzero objective).
\item \textbf{Given a starting solution, restarted PDHG converges to a \emph{nearby} optimal solution.} PDHG has the property that the distance (in terms of the PDHG norm) from the iterates to any optimal solution is nonincreasing \cite{he2014convergence}. 
The same property also holds for restarted PDHG (e.g., see Proposition~9 of \citet{applegate2023faster}).
Consequently, one can show, 
if we warm-start restarted PDHG on \eqref{primal-lp-feasibility} from a primal solution obtained by applying PDLP to the original problem (and set the starting dual iterate to zero, which is trivially dual feasible for \eqref{primal-lp-feasibility}), then restarted PDHG will 
find a nearby primal feasible solution.

Thus, if we start from a near-optimal solution (i.e., from PDLP) as a warm-start for restarted PDHG on the primal feasibility problem \eqref{primal-lp-feasibility}, then the result will be a primal solution obtaining a similar objective value which satisfies tight feasibility tolerances.
This idea is visualized in Figure~\ref{fig:visualize-feasibility-polishing}.

\end{enumerate}

\begin{figure}
\includegraphics[width=\textwidth]{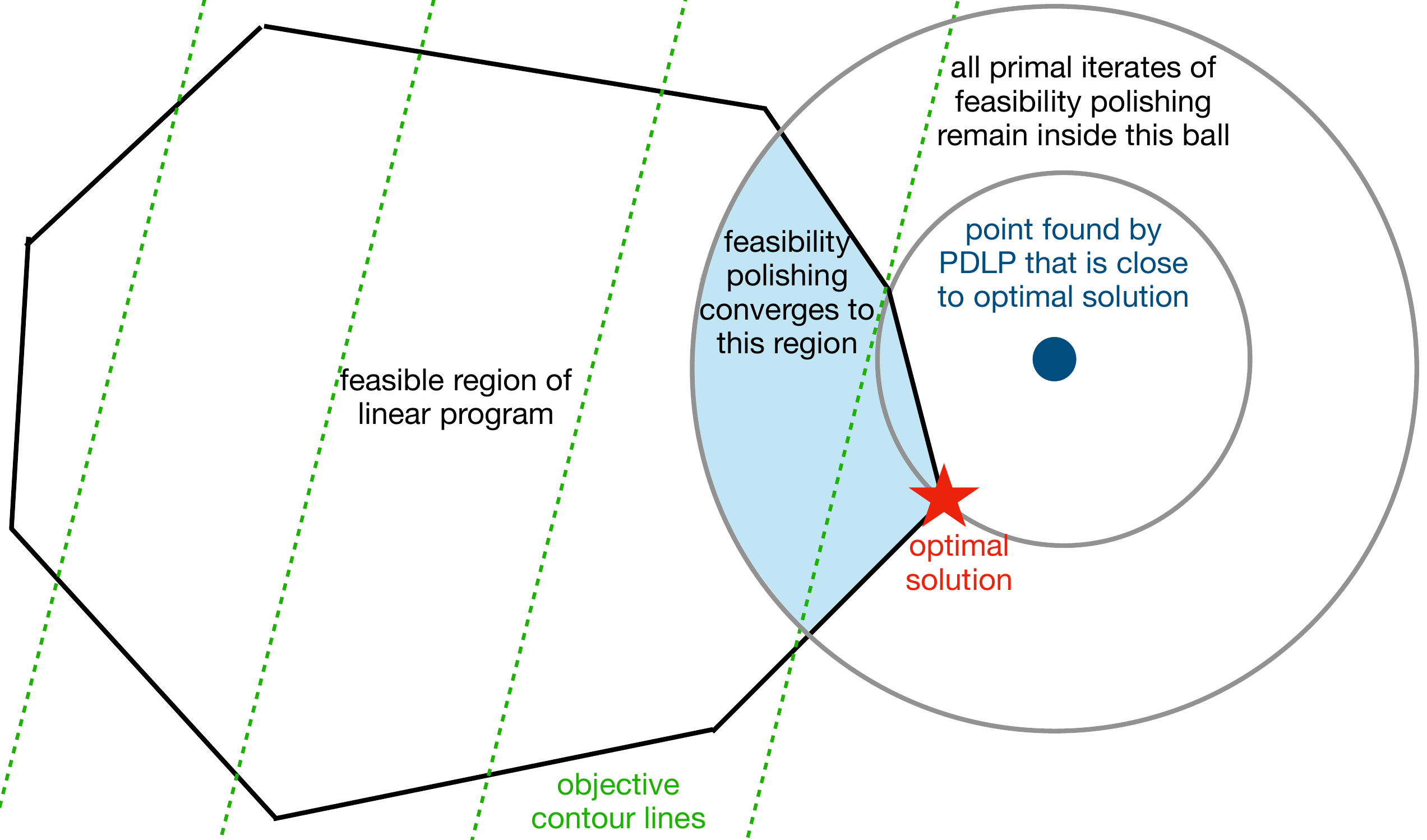}
\caption{Visualization of feasibility polishing.}\label{fig:visualize-feasibility-polishing}
\end{figure}

We can repeat this heuristic argument on the dual feasibility problem
\begin{equation}\label{dual-lp-feasibility}
    \begin{aligned}[c]
\maximize_{y\in \R^{m}, r \in\R^{n}} \quad &  0 \\
\text{subject to:}\quad & c - A^\T y = r \\
& y \in \cY \\
& r \in \cR.
    \end{aligned}
\end{equation}
to obtain a primal and dual solution that are almost exactly feasible with a small, but relatively larger, optimality gap.

Appendix~\ref{app:feas-polishing-basis} gives a theoretical basis for these two ideas. 
In particular, we show that PDHG with a fixed-frequency restart scheme and fixed step size can achieve an $O(\beta / \hoff{A} \log(1/\epsilon))$ iteration bound
on LP instances, where $\beta$ is the largest singular value of the constraint matrix $A$ and $\hoff{A}$ is its Hoffman constant \cite{hoffman_1952}.
In other words, the new theory shows a runtime bound that is constant and independent of the right-hand side or objective vector. 
In contrast, \cite{applegate2023faster} shows an $O(\beta / \hoff{K}  \log(1/\epsilon))$ iteration bound for restarted PDHG on general linear programs, where $K$ is a matrix containing the $A$ matrix, objective vector, and right-hand side.
Moreover, some dependence on the right-hand side and/or objective vector is fundamental to restarted PDHG: indeed, there are examples where $\hoff{A}=1$ but restarted PDHG is slow to converge (see \cite[Table 1]{hinder2024worst}). We also show that our polishing algorithm will converge to a nearby feasible solution.

We note that similar arguments could be applied to other first-order methods, for example, replacing PDHG with AGD and minimizing $\| A x - b \|^2$ s.t. $x \ge 0$ or using an adaptive restart scheme (following the techniques of \cite{applegate2023faster}). 
As our main focus is computational, for simplicity we analyze the setting of PDHG with fixed frequency restarts.

These two insights are employed to design Algorithm~\ref{alg:three-phase}. 
In particular, Algorithm~\ref{alg:three-phase} strives for a solution that satisfies extremely tight feasibility tolerances (e.g., $10^{-8}$ relative tolerances) but with a moderately large user-defined tolerance $\epsilon_{\text{rel-gap}} > 0$ on the relative duality gap, i.e.,
\begin{flalign}\label{eq:relative-gap}
\textsc{relative-gap} := \frac{\abs{c^{\T} x + p(-y; \lc, \uc) + p(-r ; \lv, \uv)}}{\abs{c^{\T} x} + \abs{p(-y; \lc, \uc) + p(-r ; \lv, \uv)}}  \le \epsilon_{\text{rel-gap}},
\end{flalign}
e.g., with $\epsilon_{\text{rel-gap}} = 10^{-2}$.
While this is a somewhat nonstandard relative gap definition (i.e., more typical would be a max over the two terms in the denominator), 
we use this choice for legacy reasons.

\begin{algorithm}[!htb]
\caption{Feasibility polishing algorithm.}\label{alg:three-phase}
\begin{enumerate}
\item \label{feasibility-polishing-step-normal} Run PDLP on the LP instance. After completing $k$ iterations, for $k = 100, 200, 400, 800, \dots$, if the relative duality gap is less than $\epsilon_{\text{rel-gap}}$, pause ``normal'' PDLP iterations. Let $(x^k, y^k)$ be the average iterate (since the last restart) after that $k$-th iteration.
\item \label{feasibility-polishing-step-primal} Run PDLP on the primal feasibility problem, i.e., \Cref{primal-lp-feasibility}, starting from $(x^k, 0)$ with the initial primal weight and step size equal to the current values from Step \ref{feasibility-polishing-step-normal}. Run for $k/8$ iterations or until the primal constraint violation is less than the termination tolerance. In the former case, return to Step \ref{feasibility-polishing-step-normal}, continuing ``normal'' PDLP iterations. Otherwise, let $\tilde{x}$ be the output primal solution.
\item \label{feasibility-polishing-step-dual} Run PDLP on the dual feasibility problem, i.e., \Cref{dual-lp-feasibility}, starting from $(0, y^k)$ with the initial primal weight and step size equal to the current values from Step \ref{feasibility-polishing-step-normal}. Run for $k/8$ iterations or until the dual constraint violation is less than the termination tolerance. Let $\tilde{y}$ be the output dual solution.
\item Check if $(\tilde{x}, \tilde{y})$ meets the termination criteria; if not, return to Step \ref{feasibility-polishing-step-normal}, resuming ``normal'' PDLP iterations.
\end{enumerate}
Note that operations that are triggered periodically based on iteration counts (such as termination checks, adaptive restarts --- see the paragraph ``Periodic checks of restarts and termination criteria'' in Section~\ref{sec:practical-algorithmic-improvements}) take place in Steps \ref{feasibility-polishing-step-normal}, \ref{feasibility-polishing-step-primal}, and \ref{feasibility-polishing-step-dual} using the iteration count  associated with the current Step.
\end{algorithm}

\section{Implementation}\label{sec:implementation}

The PDLP implementation we refer to in this work is a C++ implementation
that is distributed
as part of Google's open-source OR-Tools library. This implementation
of PDLP was first released with OR-Tools version 9.3 in March 2022. In this section,
we briefly discuss the motivations and notable differences between this code and the Julia-based
prototype used in the conference version of this work.

The primary motivation for a C++ implementation was the desire to 
integrate PDLP into production applications at Google, where concerns
about technical uniformity make C++ more practical to deploy than Julia.
Secondarily, the new implementation provided an opportunity to improve
efficiency and generality, given that the Julia implementation was intended as
a research prototype.

In addition to its improved accessibility thanks to its integration within both
the OR-Tools framework and external modeling systems such as CVXPY~\cite{cvxpy},
the C++ version has the following notable differences over the Julia code:

\paragraph{Feasibility polishing.} The Julia code did not have a feasibility polishing routine.

\paragraph{Multithreading.} The implementation parallelizes the computation of projections, norms, distances, and sparse matrix-dense vector multiplications by the constraint matrix $x \mapsto A x$ and its transpose $y \mapsto A^{\top} y.$ Any fixed number of threads is supported. To parallelize dense vector and vector-vector operations, it divides the vector into (approximately) equal sized shards, and performs each shard's operations independently in parallel (using Eigen for the per-shard operation). To parallelize sparse matrix-dense vector multiplications, it divides the sparse matrix by rows into shards with (approximately) the same number of nonzeros, and performs each shard's product independently in parallel (using Eigen's sparse matrix-dense vector product for the per-shard operation). Each shard is then handled independently and then combined\footnote{The algorithm is deterministic given a fixed number of shards. In each threaded operation, the per-shard operations are fully independent, so execution order doesn't matter, and the per-shard results are processed sequentially (and deterministically). Changing the number of shards changes the order of the floating-point operations, and can result in different round-off errors.}. By default, the number of shards is taken to be four times the number of threads.

\paragraph{More general LP form.} As mentioned in Section~\ref{sec:pdhg}, we consider a more general LP format, commonly adopted by LP solvers, where inequalities are permitted to have both lower and upper bounds. Equalities are specified by setting the lower bounds equal to the upper bounds. The C++ code supports this formulation, in contrast to the Julia code that
supports only single-sided inequalities and equalities.

\paragraph{Presolve.} The C++ implementation offers integrated presolving, in contrast with the conference version~\cite{applegate2021practical} where, for research purposes, we manually applied an external presolver. We did not develop new presolving routines; rather, we reused them from GLOP, the simplex-based LP solver included in OR-Tools.
As a result of reusing GLOP's presolving, the presolving code is single-threaded and not necessarily suitable for the largest scale of instances that PDLP is otherwise capable of handling.
Thus, due to the size of the instances in the numerical experiments, we do not enable presolve in this work.

\section{A New Collection of Large-Scale Instances}\label{sec:instances}

As part of this work, we developed and released a new test suite of large-scale LP instances. The set of problems is designed to be challenging for, if not beyond the capabilities of, today's state-of-the-art solvers, yet small enough to fit in memory of top-end machines. The smallest problem in our new test suite has more than 125 million nonzeros, while the largest has over 6.3 billion. The collection is intended to be both realistic and diverse, but we also acknowledge that this collection is not comprehensive.
The eleven LP instances are detailed in \Cref{table:problem-instances} and \Cref{table:problem-sizes}. Many of these problem instances were generated by the authors for the purposes of this project. There are few publicly available test instances of such a large scale, and it is our hope that these instances will be useful for other researchers.

\begin{table}[!tbh]
    \small
\begin{tabular}{p{3.4cm} p{8.7cm} p{3cm}}
\toprule
Instance name & Brief description & Full details available? \\
\midrule
design-match & \noindent Randomly generated statistical matching problem with covariate balancing constraints for making causal inference on observational data. Loosely inspired by \citet{zubizarreta2012using}. & Yes \\
tsp-gaia-10m & An LP lower bound for a travelling salesman problem. The underlying TSP instance \cite{gaia-10m-tsp} is the 10 million stars nearest to Earth in the Gaia Data Release 2 star catalog \cite{gaia-dr2_2018}. The variables correspond to 60,601,996 edges heuristically chosen to minimize the reduced-cost penalty when pricing the edges from the complete graph. The constraints are the 10,000,000 degree constraints and 7,016,681 cutting planes collected from extensive runs of the Concorde TSP solver \cite{concorde-tsp-solver} on sub-instances. & No \\
tsp-gaia-100m & Same as tsp-gaia-10m, except using the 100 million stars nearest to Earth \cite{gaia-100m-tsp}, with 1,184,557,727 edges, 100,000,000 degree constraints, and 62,934,799 cutting planes. & No \\
heat-source-easy & This problem considers a material with several randomly located heat sources. 
Given temperature measurements at various parts of the material and a list of possible locations of the heat sources, the aim is 
to recover the temperature distribution of the material. This is a PDE constrained optimization problem with linear partial differential equation constraints.  & Yes \\
heat-source-hard & Same as heat-source-easy but significantly fewer measurement locations, more heat sources and more possible heat source locations. This makes the problem harder from a recovery standpoint. Moreover, we find this problem is much harder for first-order methods than heat-source-easy. & Yes \\
production-inventory & A robust production-inventory problem \cite{ben-tal_2004}. Data is randomly generated. & Yes \\
qap-tho-150 & An LP relaxation \cite{adams_johnson_1994} of the tho150 quadratic assignment problem from QAPLIB \cite{qaplib}. Data is randomly generated. & Yes \\
qap-wil-100 & An LP relaxation \cite{adams_johnson_1994} of the wil100 quadratic assignment problem from QAPLIB \cite{qaplib}. Data is randomly generated. & Yes \\
world-shipping & An LP relaxation of a MIP formulation of the LINERLIB \cite{linerlib} World instance. & No \\
mediterranean-shipping & An LP relaxation of a MIP formulation of the LINERLIB \cite{linerlib} Mediterranean instance. & No \\
supply-chain & A multicommodity flow problem motivated by optimizing the supply chain of a large retailer. Data is randomly generated. & Yes \\
\bottomrule
\end{tabular}
\caption{Descriptions of the LP problem instances.}\label{table:problem-instances}
\end{table}

\begin{table}[!tbh]
	\small
	\center
	\begin{tabular}{l r r r}
		\toprule
		problem name & variables & constraints & nonzeros \\
		\midrule
		design-match & 40,000,000 & 22,000,135 & 2,760,000,000 \\
		tsp-gaia-10m & 60,601,996 & 17,016,681 & 475,701,996 \\
		tsp-gaia-100m & 1,184,557,727 & 162,934,799 & 6,337,834,450 \\
		heat-source-easy & 31,628,008 & 15,625,000 & 125,000,000 \\
		heat-source-hard & 31,628,008 &15,625,000& 125,000,000 \\
		production-inventory & 18,270,600 & 4,650,850 & 500,049,700 \\
		qap-tho-150 & 249,783,750 & 6,705,300 & 1,005,795,000 \\
		qap-wil-100 & 49,015,000 & 1,980,200 & 198,020,000 \\
		world-shipping & 228,867,510 & 15,304,282 & 688,658,522 \\
		mediterranean-shipping & 208,479,461 & 7,490,593 & 628,927,462 \\
		supply-chain & 201,000,100 & 2,210,100 & 403,000,100 \\
		\bottomrule
	\end{tabular}
	\caption{Problem instance sizes.}\label{table:problem-sizes}
\end{table}

All these instances were rescaled such that
\begin{enumerate}[label=(\roman*)]
	\item The objective vector has values $-1$, $0$, or $1$.
	\item The right-hand side has values $-\infty$, $-1$, $0$, $1$, or $\infty$, except in the case that a constraint has two nonzero, finite, 
	nonmatching lower and upper bounds. In this case, we rescale such that the maximum absolute value of the right hand side is one.
\end{enumerate}
This rescaling is implemented in the \textit{rescale\_instance} function in \href{https://github.com/ohinder/large-scale-LP-test-problems/blob/main/utils.jl}{utils.jl}.

These instances (after scaling) are available for download at
\[ 
\text{\url{https://www.oliverhinder.com/large-scale-lp-problems}}
\]
and some have generators available along with detailed descriptions of the problems in the repository 
\[
\text{\url{https://github.com/ohinder/large-scale-LP-test-problems}}
\]
as indicated in the right column of \Cref{table:problem-instances}.
Although, all the instances in this \emph{repository} are randomly generated (i.e., those marked with full details available),
great effort was made by the authors to make this generation as realistic as possible, 
coupled with clearly documented reasoning.

\section{Numerical experiments}\label{sec:numerical-results}

In this section we present numerical results on a new test suite of large-scale LP instances that we have developed and released as part of this work. Our goal is to show that first-order solvers, in particular PDLP with feasibility polishing, can provide feasible solutions to tight tolerances that are close-to-optimal, in reasonable time and without requiring enormous amounts of memory. We will compare PDLP with and without feasibility polishing to three commercial state-of-the-art Gurobi solvers: a barrier (i.e., interior point) method, and primal and dual simplex methods. 

We also provide some supplementary experiments in \Cref{sec:additional-numerical-experiements} on smaller test problems. These experiments (i) demonstrate 
the value of feasibility polishing on the Mittelmann benchmark and report the performance of PDLP in detecting infeasibility on a standard collection of primal infeasible problems. 
Not surprisingly we find that Gurobi is much faster than PDLP on most of these relatively small instances.

Detailed CSVs of the results and log files for all experiments can also be found at:
\[
\text{\url{https://www.oliverhinder.com/large-scale-lp-problems}}.
\]
We use the PDLP implementation in OR-Tools version 9.10.
Instructions on how to run PDLP to reproduce our results can be found at 
\[
\text{\url{https://github.com/ohinder/large-scale-LP-test-problems/blob/main/SOLVER-INSTRUCTIONS.md}.}
\]\noindent Note: we do not compare against other first-order methods or GPU implementations.
We do not compare against first-order methods, because in the conference version of this
paper \cite{applegate2021practical} we already tested against SCS, a state-of-the-art first-order method.
In these tests, PDLP was $19\times$ faster on NETLIB \cite[Table 24]{applegate2021practical} and $6.3\times$ faster on MIPLIB 2017 relaxations, and thus we do not anticipate it being competitive in our tests.
We do not compare against GPU implementations because
the release of the PDLP C++ code in 2022 pre-dated the GPU implementation of PDLP \cite{lu2025cupdlp}. Additionally, some of the instances that we test exceed the RAM limit of a single GPU.
We expect that a GPU implementation would outperform the OR-Tools code in wall clock time for instances that fit in the GPU's memory.

\paragraph{Machines used} We used the Pitt computing cluster \cite{pitt_crc} to perform our experiments. There were two types of machines that we used: (i) an AMD EPYC 7302 (Rome) processor with 16 cores (32 threads) and 256 GB of memory and (ii) an Intel Xeon Platinum 8352Y (Ice Lake) processor with 32 cores (64 threads) and 1 TB of memory. The high memory machine used two and a half times more computing credits per hour so to limit its use we first ran a solver on an instance on the machine with 256 GB of memory and then only switched to the 1 TB machine if insufficient memory was available. We estimate that our experiments would have cost around \$20,000 to run on Google cloud (excluding the cost of building and rescaling the instances). We ran PDLP with 16 threads on the 256 GB machine and 32 threads on the 1 TB machine.
For Gurobi, we left the thread parameter at its default value, which means that it will automatically use as many threads as the machine has up to the number of virtual processors available in the machine \cite{gurobi_threads}. 

\paragraph{Multithreading performance of PDLP} From Figure~\ref{fig-speed-up} we can see that PDLP enjoys significant speed-up from more threads on large instances. With $32$ threads PDLP has an almost $8 \times$ speed-up on the design-match instance, one of the largest in the set. This speed-up  comes from parallelizing the sparse matrix-vector operation and vector-vector operations. For smaller instances, the overall speed-up can be much less, in fact, more threads can even be counterproductive \cite[Table 8]{lu2025cupdlp}. We have identified two areas to improve our multithreading performance: thread synchronization and inefficient use of memory bandwidth. Implementing ideas from the literature on multithreaded sparse matrix-vector (SpMV) operations (e.g., \cite{bulucc2011reduced,karsavuran2015locality}) would lead to further improvements.

\begin{figure}[!tbh]
\center
\includegraphics[width=10cm]{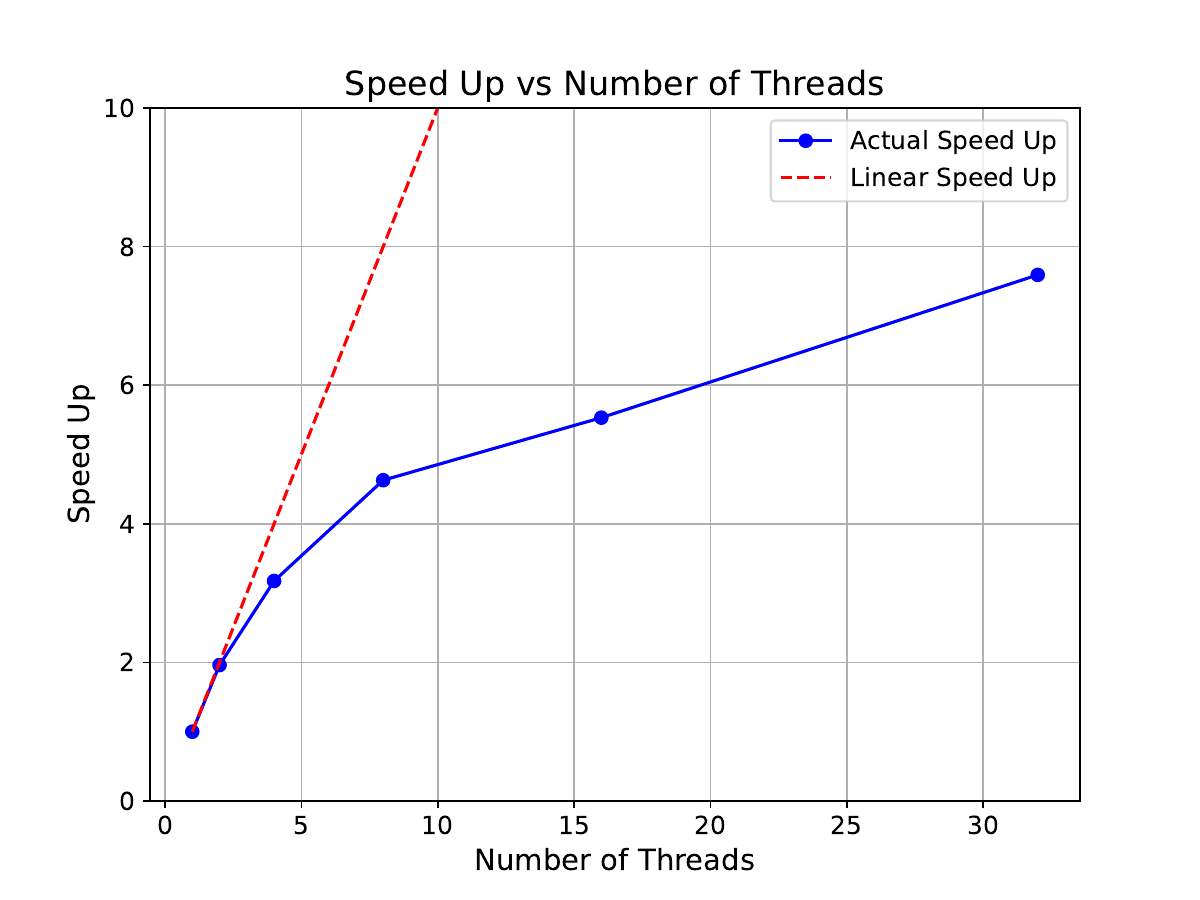}
\caption{Speed up versus number of threads for PDLP on the design-match 
linear program on an AMD EPYC 7302 (Rome) processor with 16 cores (32 threads).
We recorded the wall clock time to run 10,000 iterations with feasibility polishing disabled (1 measurement per number of threads).}\label{fig-speed-up}
\end{figure}

\paragraph{Gurobi} We use Gurobi version 11.0.2 with crossover disabled.

\paragraph{Termination criteria}
To test feasibility polishing, our experiments do 
not use the default PDLP termination criteria. Instead,
we use a tighter criterion for feasibility and a looser criterion for 
the optimality gap. In particular, in our experiments PDLP terminates when the maximum $\ell_{\infty}$ norms of the primal 
and dual residuals are less than $10^{-8}$, i.e., $x \in \cX$, $y \in \cY$, $r \in \cR$,
\begin{flalign}\label{eq:termination-criteria}
\lc - 10^{-8} \le A x \le \uc + 10^{-8}, \quad \| c - A^\T y - r \|_{\infty} \le 10^{-8},
\end{flalign}
and the relative duality gap is less than $1\%$, i.e.,
\[
\frac{\abs{c^{\T} x + p(-y; \lc, \uc) + p(-r ; \lv, \uv)}}{\abs{c^{\T} x} + \abs{p(-y; \lc, \uc) + p(-r ; \lv, \uv)}} \le 10^{-2}.
\]
Recall, however, that all linear programs have been scaled as per \Cref{sec:instances} which impacts 
the termination criteria. Specifically, \Cref{eq:termination-criteria} \emph{translated to the original linear programs} becomes equivalent to
\begin{subequations}\label{eq:rescaled-termination-criteria}
\begin{flalign}
\lc - 10^{-8} q \le A x \le \uc + 10^{-8} q ~~\text{where}~~ q_i := \begin{cases}
\abs{(\lc)_i} & (\lc)_i \not\in \{-\infty, 0 \}, \quad (\uc)_i \in \{0, \infty \} \\
\abs{(\uc)_i} & (\lc)_i \in \{-\infty, 0 \}, \quad (\uc)_i \not\in \{0, \infty \} \\
\max\{\abs{(\lc)_i}, \abs{(\uc)_i}\} & (\lc)_i \not\in \{-\infty , 0\}, \quad (\uc)_i \not\in \{0, \infty \} \\
1 & (\lc)_i \in \{-\infty , 0\}, \quad (\uc)_i \in \{0, \infty \}
\end{cases}
\end{flalign}
and
\begin{flalign}
\abs{(c - A^\T y - r )_i} \le 10^{-8} \times \begin{cases}
\abs{c_i} & c_i \neq 0 \\
1 & \text{otherwise.}
\end{cases}
\end{flalign}
\end{subequations}
Finally, we include the infeasibility termination criteria in \Cref{sec:infeasibility-termination}.

We used the default Gurobi termination criteria, which after presolve and their own custom scaling, has an $\ell_{\infty}$ tolerance on the primal and dual residuals of $10^{-6}$, and objective gap tolerance of $10^{-8}$. Note that we do not believe that making Gurobi's termination criteria match ours would make a significant difference to the results. In particular, Gurobi's barrier implementation mostly experiences OOM errors, and so is unaffected by the termination criteria; and simplex methods only find a primal and dual feasible solution at the optimum.

\paragraph{Results} Table~\ref{table:numerical-results} shows the time in \textbf{hours} taken per instance for each method to solve the problems and whether the method required a high memory machine or ran out of memory (shading the cell red).
From this table one can see that Gurobi barrier only solves three instances because it hits memory limits in all but three instances, whereas PDLP never exceeds 1 TB of memory and only in one instance requires an excess of 256 GB.
Gurobi's primal and dual simplex is competitive with PDLP in terms of memory requirements but slow: combined primal and dual simplex only solves four instances within the time limit of 144 hours.
Finally, we see that feasibility polishing is extremely effective:  
PDLP without polishing only solves two instances, PDLP with polishing solves eight of eleven instances. 
Table~\ref{table:objectives-and-duality-gaps} shows that the relative duality gap achieved by PDLP with feasibility polishing is typically much better than the $1\%$ threshold that we targeted. In particular, in all cases it obtains a duality gap of less than $0.4\%$ and other than the two shipping instances the duality gap is always less than $0.1\%$. Of course, these duality gaps are much worse than what Gurobi would achieve if it solved the problem (recall $10^{-8}$ is the default absolute optimality gap for which Gurobi barrier terminates) but for many applications this may be satisfactory.

 \Cref{table:numerical-results}  shows that both PDLP and Gurobi had errors. Both of Gurobi's errors gave the error code 10025. We believe this is a bug in Gurobi as the Gurobi error codes \cite{gurobiErrorCodes} state that the `user has called a query routine on a model with more than 2 billion non-zero entries' even though we did not call a query routine.
 We have made Gurobi aware of this issue.
Both of PDLP's errors were numerical errors that occurred within a few hundred iterations.

\begin{table}[htbp]
\centering
\begin{tabular}{lrrrrrr}
\toprule
& \multicolumn{2}{c}{\textbf{PDLP}} & \multicolumn{3}{c}{\textbf{Gurobi}} \\
\cmidrule(lr){2-3} \cmidrule(lr){4-6}
\textbf{Instance name} & \textbf{Without polishing} & \textbf{With polishing} & \textbf{Barrier} & \textbf{Dual} & \textbf{Primal} \\
\midrule
design-match & 68.0 & \textbf{9.3} & \cellcolor{red!30}OOM & \cellcolor{red!30}75.3 & \cellcolor{red!30}32.4 \\
tsp-gaia-10m & TIME & \textbf{3.0} & \cellcolor{red!30}28.1 & 90.1 & 40.2 \\
tsp-gaia-100m & TIME\cellcolor{red!30} & \cellcolor{red!30}\textbf{21.1} & \cellcolor{red!30}OOM & \cellcolor{red!30}OOM & \cellcolor{red!30}OOM \\
heat-source-easy & TIME & \textbf{60.0} & \cellcolor{red!30}OOM & ERR & TIME \\
heat-source-hard & TIME & TIME & \cellcolor{red!30}OOM & ERR & TIME \\
production-inventory & TIME & TIME & \cellcolor{red!30}\textbf{5.8} & TIME & 43.7 \\
qap-tho-150 & ERR & ERR & \cellcolor{red!30}OOM & \cellcolor{red!30}TIME & \cellcolor{red!30}TIME \\
qap-wil-100 & 44.4 & \textbf{0.3} & \cellcolor{red!30}OOM & TIME & TIME \\
world-shipping & TIME & \textbf{53.8} & \cellcolor{red!30}OOM & TIME & TIME \\
mediterranean-shipping & TIME & \textbf{32.6} & \cellcolor{red!30}OOM & TIME & TIME \\
supply-chain & TIME & 18.9 & \textbf{2.5} & 4.2 & TIME \\
\bottomrule
\end{tabular}
\caption{Comparison of solvers in terms of the number of \textbf{hours} to solve instances
(excluding time to read the instance).
Termination criteria for Gurobi was set as default; termination criteria for PDLP was $10^{-8}$ maximum violation on primal and dual feasibility but $10^{-2}$ on the relative duality gap. 
Cells shaded red indicate runs that required the 1 TB RAM machine. TIME indicates that the solver exceeded 144 hours. OOM stands for `out of memory', meaning the solver exceeded the 1 TB RAM limit. ERR denotes that there was an error during the solve.}\label{table:numerical-results}
\end{table}

\begin{table}
\centering
\begin{tabular}{lrrl}
\toprule
\textbf{Instance name} & \textbf{Primal objective} &  \textbf{Absolute gap}  & \textbf{Relative gap} (\%)  \\
\midrule
design-match & $2.87 \times 10^{6}$ & $5.53 \times 10^{-2}$ & $0.0000001$ \\
tsp-gaia-10m & $1.3 \times 10^{8}$ & $3.01 \times 10^3$ & $0.0012$ \\
tsp-gaia-100m & $1.94 \times 10^{9}$  & $5.82 \times 10^{4}$ & $0.0015$   \\
heat-source-easy & $5.16 \times 10^{7}$ & $9.06 \times 10^{-4}$ & $0.00000000088$ \\
qap-wil-100 & $2.21 \times 10^5$ & $1.14 \times 10^{2}$ & $0.026$ \\
world-shipping & $-1.18 \times 10^{5}$ & $3.7 \times 10^{2}$ & $0.16$ \\
mediterranean-shipping & $-1.07 \times 10^{3}$ & $6.95 \times 10^{0}$ & $0.33$ \\
supply-chain & $1.08 \times 10^8$ & $6.67 \times 10^3$ & $0.0031$ \\
\bottomrule
\end{tabular}
\caption{Objectives and duality gaps produced from PDLP with feasibility polishing for instances that successfully terminated.
Note the relative gap is computed with the 
somewhat non-standard formula given in \Cref{eq:relative-gap}.
}\label{table:objectives-and-duality-gaps}
\end{table}

\section{Conclusion}

In this work, we demonstrated that feasibility polishing significantly improves the ability of PDLP to find solutions with tight feasibility tolerances (both primal and dual) but more relaxed optimality tolerances.
Practitioners would obviously prefer solutions that satisfy tight feasibility and optimality tolerances—solutions which are reliably provided by interior point methods and the simplex method on moderately large instances. 
However, the standard strategy of relaxing both feasibility and optimality tolerances to allow first-order methods to produce solutions within a reasonable timeframe \cite{o2016conic} is undesirable: moderate feasibility violations can give unexpectedly catastrophic solutions to the true underlying problem. 
This issue is not merely hypothetical; it occurs in the `irish-electricity' instance as discussed in the introduction. Moreover, ruling out this issue requires a solution with tight feasibility tolerances, which defeats the purpose of using a first-order method for a fast, inaccurate solve.

We also showed, on a new collection of very large linear programs and a machine with 1TB of RAM, that PDLP with feasibility polishing can find solutions with extremely small feasibility violations and moderate duality gaps for several problems that Gurobi is unable to solve within six days. 
While we expect our work to expand the number of large-scale linear programs that can be solved, it is important to emphasize the inherent limitations of first-order methods: they are unlikely to ever completely replace second-order methods.
In particular, one can construct linear systems with binary entries and condition number exponential in the dimension \cite{alon1997anti}. 
Such linear systems are unlikely to be solved efficiently by first-order methods, which in the worst case require a number of iterations polynomial in the condition number to obtain high-accuracy solutions.

Nonetheless, our numerical results illustrate that large sets of practical linear programs can be effectively tackled with first-order methods. 
Indeed, when constraint matrices are totally unimodular \cite{hinder2024worst} or have small circuit imbalance \cite{cole2025first}, first-order methods provably require only a polynomial (in the problem dimensions) number of iterations to find solutions satisfying the optimality conditions with tight tolerances. 
Relaxing the termination criteria to require only tight feasibility tolerances with moderate duality gap tolerances will likely expand this set even further.

\newpage 

\section*{Acknowledgements}

We thank Ed Klotz for his helpful discussion regarding the Irish Electricity instance and his help clarifying the 
termination criteria of Gurobi.
We also thank Joey Huchette for his feedback on an early draft of the paper.

This research was supported in part by the University of Pittsburgh Center for Research Computing, RRID:SCR\_022735, through the resources provided. Specifically, this work used the SMP cluster, which is supported by NSF award number OAC-2117681.

\section*{Author contributions}

Authors are ordered alphabetically. All authors contributed to writing, and regular discussion in meetings. The listed contributions span both the conference and journal papers.

DA and WS jointly designed the feasibility polishing implementation, and implemented most of the C++ code (except for the parts specifically mentioned elsewhere).

DA led the project after ML left Google, and generated the large-scale TSP instances which were our first large-scale test problems.

MD analyzed and implemented infeasibility detection features in both Julia and C++. He implemented and evaluated our step size choice against Malitsky and Pock's. He also proofread and evaluated the large instances designed by OH.

OH devised the restart and primal weight scheme, implementing them in FirstOrderLp.jl; he proposed, and demonstrated the feasibility polishing concept. He also created the design-match, supply-chain and heat-source test problems. DA and WS ran preliminary experiments for this paper and OH ran the final experiments.

HL's internship provided the impetus and initial code to pursue the idea of using first-order methods for LP. After his internship, HL, jointly with BOD, proposed the use of PDHG as the base algorithm. He designed and implemented the preconditioning scheme in both the FirstOrderLp.jl and C++ code. He also contributed to the development of the theory for restart strategies and infeasibility detection which influenced the implementation, and he implemented the production-inventory problem.

ML led the project while he was at Google, recruited and supervised HL, OH, and MD for their work on the project that was conducted at Google, contributed significantly to the FirstOrderLp.jl and OR-Tools implementations, created a framework for running large-scale experiments, suggested many aspects of the design of the numerical experiments, pushed forward the presolve enhancement, and led the open sourcing effort.

BOD, jointly with HL, proposed the usage of PDHG as the core algorithm to apply to LP. He also provided advice and feedback on the restarting techniques and the usage of SCS, and helped with data analysis and figure generation for the NeurIPS paper.

WS performed the initial experiments demonstrating the effectiveness of mirror-prox and PDHG for LP including some initial heuristics that inspired the primal weight scheme and restarts, and designed and implemented the sharded parallelism mechanism and the adaptive step size scheme.

\bibliographystyle{plainnat}
\bibliography{main}

\begin{thebibliography}{100}
\providecommand{\natexlab}[1]{#1}
\providecommand{\url}[1]{\texttt{#1}}
\expandafter\ifx\csname urlstyle\endcsname\relax
  \providecommand{\doi}[1]{doi: #1}\else
  \providecommand{\doi}{doi: \begingroup \urlstyle{rm}\Url}\fi

\bibitem[gur(2024)]{gurobi_threads}
Gurobi optimizer reference manual - threads parameter.
\newblock \url{https://www.gurobi.com/documentation/current/refman/threads.html\#parameter:Threads}, 2024.
\newblock Accessed: 2024-09-14.

\bibitem[pit(2024)]{pitt_crc}
Center for research computing.
\newblock \url{https://crc.pitt.edu/}, 2024.
\newblock Accessed: 2024-09-14.

\bibitem[qap(2024)]{qaplib}
{QAPLIB} problem instances and solutions.
\newblock \url{https://coral.ise.lehigh.edu/data-sets/qaplib/qaplib-problem-instances-and-solutions/}, 2024.
\newblock Accessed: 2024-09-14.

\bibitem[Adams and Johnson(1994)]{adams_johnson_1994}
Warren~P. Adams and Terri~A. Johnson.
\newblock Improved linear programming-based lower bounds for the quadratic assignment problem.
\newblock In Panos Pardalos and Henry Wolkowicz, editors, \emph{Quadratic Assignment and Related Problems}, volume~16 of \emph{DIMACS Series on Discrete Mathematics and Theoretical Computer Science}, pages 43--75. AMS, Providence, RI, 1994.

\bibitem[Alacaoglu et~al.(2022)Alacaoglu, Fercoq, and Cevher]{alacaoglu2019convergence}
Ahmet Alacaoglu, Olivier Fercoq, and Volkan Cevher.
\newblock On the convergence of stochastic primal-dual hybrid gradient.
\newblock \emph{SIAM Journal on Optimization}, 2022.

\bibitem[Alon and V{\~u}(1997)]{alon1997anti}
Noga Alon and V{\u{a}}n~H V{\~u}.
\newblock Anti-{H}adamard matrices, coin weighing, threshold gates, and indecomposable hypergraphs.
\newblock \emph{Journal of Combinatorial Theory, Series A}, 79\penalty0 (1):\penalty0 133--160, 1997.

\bibitem[Andreani et~al.(2008)Andreani, Birgin, Mart{\'\i}nez, and Schuverdt]{andreani2008augmented}
Roberto Andreani, Ernesto~G Birgin, Jos{\'e}~Mario Mart{\'\i}nez, and Mar{\'\i}a~Laura Schuverdt.
\newblock On augmented {L}agrangian methods with general lower-level constraints.
\newblock \emph{SIAM Journal on Optimization}, 18\penalty0 (4):\penalty0 1286--1309, 2008.

\bibitem[Applegate et~al.(2020)Applegate, Bixby, Chv\'{a}tal, and Cook]{concorde-tsp-solver}
David Applegate, Robert Bixby, Va\v{s}ek Chv\'{a}tal, and William Cook.
\newblock Concorde {TSP} solver.
\newblock \url{https://www.math.uwaterloo.ca/tsp/concorde.html}, 2020.

\bibitem[Applegate et~al.(2021)Applegate, D{\'\i}az, Hinder, Lu, Lubin, O'Donoghue, and Schudy]{applegate2021practical}
David Applegate, Mateo D{\'\i}az, Oliver Hinder, Haihao Lu, Miles Lubin, Brendan O'Donoghue, and Warren Schudy.
\newblock Practical large-scale linear programming using primal-dual hybrid gradient.
\newblock \emph{Advances in Neural Information Processing Systems}, 34:\penalty0 20243--20257, 2021.

\bibitem[Applegate et~al.(2023)Applegate, Hinder, Lu, and Lubin]{applegate2023faster}
David Applegate, Oliver Hinder, Haihao Lu, and Miles Lubin.
\newblock Faster first-order primal-dual methods for linear programming using restarts and sharpness.
\newblock \emph{Mathematical Programming}, 201\penalty0 (1):\penalty0 133--184, 2023.

\bibitem[Applegate et~al.(2024)Applegate, D{\'\i}az, Lu, and Lubin]{applegate2024infeasibility}
David Applegate, Mateo D{\'\i}az, Haihao Lu, and Miles Lubin.
\newblock Infeasibility detection with primal-dual hybrid gradient for large-scale linear programming.
\newblock \emph{SIAM Journal on Optimization}, 34\penalty0 (1):\penalty0 459--484, 2024.

\bibitem[Basu et~al.(2020)Basu, Ghoting, Mazumder, and Pan]{basu2020eclipse}
Kinjal Basu, Amol Ghoting, Rahul Mazumder, and Yao Pan.
\newblock {ECLIPSE}: An extreme-scale linear program solver for web-applications.
\newblock In \emph{International Conference on Machine Learning}, pages 704--714. PMLR, 2020.

\bibitem[Bauschke and Combettes(2017)]{bauschke2011convex}
Heinz~H Bauschke and Patrick~L Combettes.
\newblock \emph{Convex analysis and monotone operator theory in {H}ilbert spaces}, volume 408.
\newblock Springer, 2nd edition, 2017.

\bibitem[Beck and Teboulle(2009)]{beck2009fast}
Amir Beck and Marc Teboulle.
\newblock A fast iterative shrinkage-thresholding algorithm for linear inverse problems.
\newblock \emph{SIAM Journal on Imaging Sciences}, 2\penalty0 (1):\penalty0 183--202, 2009.

\bibitem[Becker et~al.(2011)Becker, Cand{\`e}s, and Grant]{becker2011templates}
Stephen~R Becker, Emmanuel~J Cand{\`e}s, and Michael~C Grant.
\newblock Templates for convex cone problems with applications to sparse signal recovery.
\newblock \emph{Mathematical Programming Computation}, 3\penalty0 (3):\penalty0 165--218, 2011.

\bibitem[Ben-Tal et~al.(2004)Ben-Tal, Goryashko, Guslitzer, and Nemirovski]{ben-tal_2004}
Aharon Ben-Tal, Alexander Goryashko, Elana Guslitzer, and Arkadi Nemirovski.
\newblock Adjustable robust solutions of uncertain linear programs.
\newblock \emph{Mathematical Programming}, 99\penalty0 (2):\penalty0 351--376, 2004.

\bibitem[Birgin and Mart{\'\i}nez(2002)]{birgin2002large}
Ernesto~G Birgin and Jos{\'e}~Mario Mart{\'\i}nez.
\newblock Large-scale active-set box-constrained optimization method with spectral projected gradients.
\newblock \emph{Computational Optimization and Applications}, 23\penalty0 (1):\penalty0 101--125, 2002.

\bibitem[Blin(2024)]{cuopt}
Nicolas Blin.
\newblock Accelerate large linear programming problems with {NVIDIA} cu{O}pt.
\newblock \url{https://developer.nvidia.com/blog/accelerate-large-linear-programming-problems-with-nvidia-cuopt/}, 2024.

\bibitem[Boyd and Vandenberghe(2004)]{boyd2004convex}
Stephen Boyd and Lieven Vandenberghe.
\newblock \emph{Convex optimization}.
\newblock Cambridge university press, 2004.

\bibitem[Boyd et~al.(2011)Boyd, Parikh, Chu, Peleato, and Eckstein]{boyd2011distributed}
Stephen Boyd, Neal Parikh, Eric Chu, Borja Peleato, and Jonathan Eckstein.
\newblock Distributed optimization and statistical learning via the alternating direction method of multipliers.
\newblock \emph{Foundations and Trends{\textregistered} in Machine learning}, 3\penalty0 (1):\penalty0 1--122, 2011.

\bibitem[Brin and Page(1998)]{brin1998anatomy}
Sergey Brin and Lawrence Page.
\newblock The anatomy of a large-scale hypertextual web search engine.
\newblock \emph{Computer networks and ISDN systems}, 30\penalty0 (1-7):\penalty0 107--117, 1998.

\bibitem[Brouer(2013)]{linerlib}
Berit Brouer.
\newblock {LINERLIB}.
\newblock \url{https://github.com/blof/LINERLIB}, 2013.

\bibitem[Brown et~al.(2018)Brown, Vallenari, Prusti, et~al.]{gaia-dr2_2018}
A.~G.~A. Brown, A.~Vallenari, T.~Prusti, et~al.
\newblock Gaia data release 2: Summary of the contents and survey properties.
\newblock \emph{Astronomy \& Astrophysics}, 616:\penalty0 A1, August 2018.
\newblock \doi{10.1051/0004-6361/201833051}.

\bibitem[Bulu{\c{c}} et~al.(2011)Bulu{\c{c}}, Williams, Oliker, and Demmel]{bulucc2011reduced}
Aydin Bulu{\c{c}}, Samuel Williams, Leonid Oliker, and James Demmel.
\newblock Reduced-bandwidth multithreaded algorithms for sparse matrix-vector multiplication.
\newblock In \emph{2011 IEEE International Parallel \& Distributed Processing Symposium}, pages 721--733. IEEE, 2011.

\bibitem[Carroll et~al.(2017)Carroll, Flynn, Fortz, and Melhorn]{carroll2017sub}
Paula Carroll, Damian Flynn, Bernard Fortz, and Alex Melhorn.
\newblock Sub-hour unit commitment {MILP} model with benchmark problem instances.
\newblock In \emph{Computational Science and Its Applications--ICCSA 2017: 17th International Conference, Trieste, Italy, July 3-6, 2017, Proceedings, Part II 17}, pages 635--651. Springer, 2017.

\bibitem[Chambolle and Pock(2011)]{chambolle2011first}
Antonin Chambolle and Thomas Pock.
\newblock A first-order primal-dual algorithm for convex problems with applications to imaging.
\newblock \emph{Journal of Mathematical Imaging and Vision}, 40\penalty0 (1):\penalty0 120--145, 2011.

\bibitem[Chambolle and Pock(2016)]{chambolle2016ergodic}
Antonin Chambolle and Thomas Pock.
\newblock On the ergodic convergence rates of a first-order primal--dual algorithm.
\newblock \emph{Mathematical Programming}, 159\penalty0 (1):\penalty0 253--287, 2016.

\bibitem[Chambolle et~al.(2018)Chambolle, Ehrhardt, Richt{\'a}rik, and Schonlieb]{chambolle2018stochastic}
Antonin Chambolle, Matthias~J Ehrhardt, Peter Richt{\'a}rik, and Carola-Bibiane Schonlieb.
\newblock Stochastic primal-dual hybrid gradient algorithm with arbitrary sampling and imaging applications.
\newblock \emph{SIAM Journal on Optimization}, 28\penalty0 (4):\penalty0 2783--2808, 2018.

\bibitem[Cole et~al.(2025)Cole, Hertrich, Tao, and V{\'e}gh]{cole2025first}
Richard Cole, Christoph Hertrich, Yixin Tao, and L{\'a}szl{\'o}~A V{\'e}gh.
\newblock A first order method for linear programming parameterized by circuit imbalance.
\newblock \emph{Mathematical Programming}, pages 1--39, 2025.

\bibitem[Condat(2013)]{condat2013primal}
Laurent Condat.
\newblock A primal--dual splitting method for convex optimization involving {L}ipschitzian, proximable and linear composite terms.
\newblock \emph{Journal of Optimization Theory and Applications}, 158\penalty0 (2):\penalty0 460--479, 2013.

\bibitem[Cook(2024{\natexlab{a}})]{gaia-100m-tsp}
William Cook.
\newblock 100,000,000 stars.
\newblock \url{https://www.math.uwaterloo.ca/tsp/star/star100m.html}, 2024{\natexlab{a}}.

\bibitem[Cook(2024{\natexlab{b}})]{gaia-10m-tsp}
William Cook.
\newblock 10,000,000 stars.
\newblock \url{https://www.math.uwaterloo.ca/tsp/star/star10m.html}, 2024{\natexlab{b}}.

\bibitem[Cui et~al.(2019)Cui, Morikuni, Tsuchiya, and Hayami]{cui2019implementation}
Yiran Cui, Keiichi Morikuni, Takashi Tsuchiya, and Ken Hayami.
\newblock Implementation of interior-point methods for {LP} based on {K}rylov subspace iterative solvers with inner-iteration preconditioning.
\newblock \emph{Computational Optimization and Applications}, 74\penalty0 (1):\penalty0 143--176, 2019.

\bibitem[Deng et~al.(2025)Deng, Feng, Gao, Ge, Jiang, Jiang, Liu, Liu, Xue, Ye, and Zhang]{deng2022new}
Qi~Deng, Qing Feng, Wenzhi Gao, Dongdong Ge, Bo~Jiang, Yuntian Jiang, Jingsong Liu, Tianhao Liu, Chenyu Xue, Yinyu Ye, and Chuwen Zhang.
\newblock An enhanced alternating direction method of multipliers-based interior point method for linear and conic optimization.
\newblock \emph{INFORMS Journal on Computing}, 37\penalty0 (2):\penalty0 338--359, 2025.

\bibitem[Diamond and Boyd(2016)]{cvxpy}
Steven Diamond and Stephen~P. Boyd.
\newblock {CVXPY}: {A} python-embedded modeling language for convex optimization.
\newblock \emph{Journal of Machine Learning Research}, 17, 2016.

\bibitem[Duchi et~al.(2011)Duchi, Hazan, and Singer]{duchi2011adaptive}
John Duchi, Elad Hazan, and Yoram Singer.
\newblock Adaptive subgradient methods for online learning and stochastic optimization.
\newblock \emph{Journal of Machine Learning Research}, 12:\penalty0 2121--2159, 2011.

\bibitem[Eckstein and Bertsekas(1990)]{eckstein1990alternating}
Jonathan Eckstein and Dimitri~P Bertsekas.
\newblock An alternating direction method for linear programming.
\newblock Technical Report LIDS-P-1967, Laboratory for Information and Decision Systems, Massachusetts Institute of Technology, 1990.

\bibitem[Eckstein and Bertsekas(1992)]{eckstein1992douglas}
Jonathan Eckstein and Dimitri~P Bertsekas.
\newblock On the {D}ouglas--{R}achford splitting method and the proximal point algorithm for maximal monotone operators.
\newblock \emph{Mathematical Programming}, 55\penalty0 (1-3):\penalty0 293--318, 1992.

\bibitem[Esser et~al.(2010)Esser, Zhang, and Chan]{esser2010general}
Ernie Esser, Xiaoqun Zhang, and Tony~F Chan.
\newblock A general framework for a class of first order primal-dual algorithms for convex optimization in imaging science.
\newblock \emph{{SIAM} Journal on Imaging Sciences}, 3\penalty0 (4):\penalty0 1015--1046, 2010.

\bibitem[Fountoulakis et~al.(2014)Fountoulakis, Gondzio, and Zhlobich]{fountoulakis2014matrix}
Kimon Fountoulakis, Jacek Gondzio, and Pavel Zhlobich.
\newblock Matrix-free interior point method for compressed sensing problems.
\newblock \emph{Mathematical Programming Computation}, 6\penalty0 (1):\penalty0 1--31, 2014.

\bibitem[Garstka et~al.(2021)Garstka, Cannon, and Goulart]{garstka_2019}
Michael Garstka, Mark Cannon, and Paul Goulart.
\newblock {COSMO}: A conic operator splitting method for convex conic problems.
\newblock \emph{Journal of Optimization Theory and Applications}, 190\penalty0 (3):\penalty0 779--810, 2021.

\bibitem[Gay(1985)]{netlib}
David~M Gay.
\newblock Electronic mail distribution of linear programming test problems.
\newblock \emph{Mathematical Programming Society COAL Newsletter}, 13:\penalty0 10--12, 1985.

\bibitem[Ge et~al.(2024)Ge, Huangfu, Wang, Wu, and Ye]{copt}
Dongdong Ge, Qi~Huangfu, Zizhuo Wang, Jian Wu, and Yinyu Ye.
\newblock Cardinal {O}ptimizer {(COPT)} user guide.
\newblock \url{https://guide.coap.online/copt/en-doc}, 2024.

\bibitem[Gilpin et~al.(2012)Gilpin, Pe\~{n}a, and Sandholm]{gilpin2012first}
Andrew Gilpin, Javier Pe\~{n}a, and Tuomas Sandholm.
\newblock First-order algorithm with $\mathcal{O}(\ln(1/\epsilon))$ convergence for $\epsilon$-equilibrium in two-person zero-sum games.
\newblock \emph{Mathematical Programming}, 133\penalty0 (1):\penalty0 279--298, 2012.

\bibitem[Goldstein et~al.(2013)Goldstein, Li, Yuan, Esser, and Baraniuk]{goldstein2013adaptive}
Tom Goldstein, Min Li, Xiaoming Yuan, Ernie Esser, and Richard Baraniuk.
\newblock Adaptive primal-dual hybrid gradient methods for saddle-point problems.
\newblock \emph{arXiv:1305.0546}, 2013.

\bibitem[Goldstein et~al.(2015)Goldstein, Li, and Yuan]{goldstein2015adaptive}
Tom Goldstein, Min Li, and Xiaoming Yuan.
\newblock Adaptive primal-dual splitting methods for statistical learning and image processing.
\newblock In \emph{Advances in Neural Information Processing Systems}, pages 2089--2097, 2015.

\bibitem[Gondzio(2012)]{gondzio2012}
Jacek Gondzio.
\newblock Matrix-free interior point method.
\newblock \emph{Computational Optimization and Applications}, 51\penalty0 (2):\penalty0 457--480, 2012.

\bibitem[Gondzio et~al.(2022)Gondzio, Pougkakiotis, and Pearson]{gondzio2022general}
Jacek Gondzio, Spyridon Pougkakiotis, and John~W Pearson.
\newblock General-purpose preconditioning for regularized interior point methods.
\newblock \emph{Computational Optimization and Applications}, 83\penalty0 (3):\penalty0 727--757, 2022.

\bibitem[{Gurobi Optimization, LLC}(2024)]{gurobiErrorCodes}
{Gurobi Optimization, LLC}.
\newblock \emph{Gurobi Optimizer Reference Manual: Error Codes}, 2024.
\newblock URL \url{https://www.gurobi.com/documentation/current/refman/error_codes.html}.
\newblock Accessed: 2024-09-13.

\bibitem[Halpern(1967)]{halpern1967fixed}
Benjamin Halpern.
\newblock Fixed points of nonexpanding maps.
\newblock \emph{Bulletin of the American Mathematical Society}, 1967.

\bibitem[He and Yuan(2012{\natexlab{a}})]{he2012convergence}
Bingsheng He and Xiaoming Yuan.
\newblock Convergence analysis of primal-dual algorithms for a saddle-point problem: from contraction perspective.
\newblock \emph{SIAM Journal on Imaging Sciences}, 5\penalty0 (1):\penalty0 119--149, 2012{\natexlab{a}}.

\bibitem[He and Yuan(2012{\natexlab{b}})]{heyuan2012}
Bingsheng He and Xiaoming Yuan.
\newblock On the ${O}(1/n)$ convergence rate of the {D}ouglas--{R}achford alternating direction method.
\newblock \emph{{SIAM} Journal on Numerical Analysis}, 50\penalty0 (2):\penalty0 700--709, 2012{\natexlab{b}}.

\bibitem[He et~al.(2014)He, You, and Yuan]{he2014convergence}
Bingsheng He, Yanfei You, and Xiaoming Yuan.
\newblock On the convergence of primal-dual hybrid gradient algorithm.
\newblock \emph{SIAM Journal on Imaging Sciences}, 7\penalty0 (4):\penalty0 2526--2537, 2014.

\bibitem[Hinder(2024)]{hinder2024worst}
Oliver Hinder.
\newblock Worst-case analysis of restarted primal-dual hybrid gradient on totally unimodular linear programs.
\newblock \emph{Operations Research Letters}, 57:\penalty0 107199, 2024.

\bibitem[Hoffman(1952)]{hoffman_1952}
Alan~J. Hoffman.
\newblock On approximate solutions of systems of linear inequalities.
\newblock \emph{Journal of Research of the National Bureau of Standards}, 49:\penalty0 263--265, 1952.

\bibitem[Huangfu and Hall(2018)]{huangfu2018parallelizing}
Qi~Huangfu and J.~A.~J. Hall.
\newblock Parallelizing the dual revised simplex method.
\newblock \emph{Mathematical Programming Computation}, 10\penalty0 (1):\penalty0 119--142, 2018.

\bibitem[Kaplan et~al.(2020)Kaplan, McCandlish, Henighan, Brown, Chess, Child, Gray, Radford, Wu, and Amodei]{kaplan2020scaling}
Jared Kaplan, Sam McCandlish, Tom Henighan, Tom~B Brown, Benjamin Chess, Rewon Child, Scott Gray, Alec Radford, Jeffrey Wu, and Dario Amodei.
\newblock Scaling laws for neural language models.
\newblock \emph{arXiv:2001.08361}, 2020.

\bibitem[Karsavuran et~al.(2015)Karsavuran, Akbudak, and Aykanat]{karsavuran2015locality}
M~Ozan Karsavuran, Kadir Akbudak, and Cevdet Aykanat.
\newblock Locality-aware parallel sparse matrix-vector and matrix-transpose-vector multiplication on many-core processors.
\newblock \emph{IEEE Transactions on Parallel and Distributed Systems}, 27\penalty0 (6):\penalty0 1713--1726, 2015.

\bibitem[Krizhevsky et~al.(2012)Krizhevsky, Sutskever, and Hinton]{krizhevsky2012imagenet}
Alex Krizhevsky, Ilya Sutskever, and Geoffrey~E Hinton.
\newblock Imagenet classification with deep convolutional neural networks.
\newblock \emph{Advances in neural information processing systems}, 25, 2012.

\bibitem[Lan et~al.(2011)Lan, Lu, and Monteiro]{Lan2011}
Guanghui Lan, Zhaosong Lu, and Renato D.~C. Monteiro.
\newblock Primal-dual first-order methods with $\mathcal{O}(1/\epsilon)$ iteration-complexity for cone programming.
\newblock \emph{Mathematical Programming}, 2011.

\bibitem[Lin et~al.(2021)Lin, Ma, Ye, and Zhang]{lin2021admm}
Tianyi Lin, Shiqian Ma, Yinyu Ye, and Shuzhong Zhang.
\newblock An {ADMM}-based interior-point method for large-scale linear programming.
\newblock \emph{Optimization Methods and Software}, 36\penalty0 (2-3):\penalty0 389--424, 2021.

\bibitem[Lu and Yang(2022)]{lu2022infimal}
Haihao Lu and Jinwen Yang.
\newblock On the infimal sub-differential size of primal-dual hybrid gradient method and beyond.
\newblock \emph{arXiv:2206.12061}, 2022.

\bibitem[Lu and Yang(2023)]{lu2023unified}
Haihao Lu and Jinwen Yang.
\newblock On a unified and simplified proof for the ergodic convergence rates of {PPM}, {PDHG} and {ADMM}.
\newblock \emph{arXiv:2305.02165}, 2023.

\bibitem[Lu and Yang(2024)]{lu2024restarted}
Haihao Lu and Jinwen Yang.
\newblock Restarted {H}alpern {PDHG} for linear programming.
\newblock \emph{arXiv:2407.16144}, 2024.

\bibitem[Lu and Yang(2025{\natexlab{a}})]{lu2024geometry}
Haihao Lu and Jinwen Yang.
\newblock On the geometry and refined rate of primal--dual hybrid gradient for linear programming.
\newblock \emph{Mathematical Programming}, 212:\penalty0 349--387, 2025{\natexlab{a}}.

\bibitem[Lu and Yang(2025{\natexlab{b}})]{lu2025cupdlp}
Haihao Lu and Jinwen Yang.
\newblock {cuPDLP.jl}: A {GPU} implementation of restarted primal-dual hybrid gradient for linear programming in julia.
\newblock \emph{Operations Research}, 73\penalty0 (6):\penalty0 3440--3452, 2025{\natexlab{b}}.

\bibitem[Lu et~al.(2023)Lu, Yang, Hu, Huangfu, Liu, Liu, Ye, Zhang, and Ge]{lu2023cupdlpc}
Haihao Lu, Jinwen Yang, Haodong Hu, Qi~Huangfu, Jinsong Liu, Tianhao Liu, Yinyu Ye, Chuwen Zhang, and Dongdong Ge.
\newblock {cuPDLP-C}: A strengthened implementation of {cuPDLP} for linear programming by {C} language.
\newblock \emph{arXiv:2312.14832}, 2023.

\bibitem[Maros(2011)]{maros_simplex_encyclopedia}
István Maros.
\newblock \emph{Simplex-Based {LP} Solvers}.
\newblock John Wiley \& Sons, Ltd, 2011.
\newblock ISBN 9780470400531.

\bibitem[McMahan and Streeter(2010)]{mcmahan2010adaptive}
H~Brendan McMahan and Matthew Streeter.
\newblock Adaptive bound optimization for online convex optimization.
\newblock \emph{arXiv:1002.4908}, 2010.

\bibitem[MIPLIB(2017)]{miplib_irish_electricity}
MIPLIB.
\newblock {Irish Electricity Instance - MIPLIB 2017}, 2017.
\newblock URL \url{https://miplib.zib.de/instance_details_irish-electricity.html}.
\newblock Accessed: 2024-10-14.

\bibitem[Mittelmann(2021)]{mittelmannbenchmark}
Hans~D Mittelmann.
\newblock Decision tree for optimization software.
\newblock \url{http://plato.asu.edu/guide.html}, 2021.

\bibitem[Necoara et~al.(2019)Necoara, Nesterov, and Glineur]{necoara2019linearfom}
Ion Necoara, Yu~Nesterov, and Francois Glineur.
\newblock Linear convergence of first order methods for non-strongly convex optimization.
\newblock \emph{Mathematical Programming}, 2019.

\bibitem[Nesterov(1983)]{nesterov1983method}
Yurii Nesterov.
\newblock A method of solving a convex programming problem with convergence rate {$O(1/k^2)$}.
\newblock \emph{Soviet Mathematics Doklady}, 27\penalty0 (2):\penalty0 372--376, 1983.

\bibitem[O'Donoghue(2021)]{o2021operator}
Brendan O'Donoghue.
\newblock Operator splitting for a homogeneous embedding of the linear complementarity problem.
\newblock \emph{SIAM Journal on Optimization}, 31\penalty0 (3):\penalty0 1999--2023, 2021.

\bibitem[O'Donoghue et~al.(2016)O'Donoghue, Chu, Parikh, and Boyd]{o2016conic}
Brendan O'Donoghue, Eric Chu, Neal Parikh, and Stephen Boyd.
\newblock Conic optimization via operator splitting and homogeneous self-dual embedding.
\newblock \emph{Journal of Optimization Theory and Applications}, 169\penalty0 (3):\penalty0 1042--1068, 2016.

\bibitem[Optimization()]{gurobi_mipgap}
Gurobi Optimization.
\newblock What is the {MIPG}ap?
\newblock \url{https://support.gurobi.com/hc/en-us/articles/8265539575953-What-is-the-MIPGap}.
\newblock Accessed: 2024-10-07.

\bibitem[O’Connor and Vandenberghe(2020)]{o2020equivalence}
Daniel O’Connor and Lieven Vandenberghe.
\newblock On the equivalence of the primal-dual hybrid gradient method and {D}ouglas--{R}achford splitting.
\newblock \emph{Mathematical Programming}, 179\penalty0 (1):\penalty0 85--108, 2020.

\bibitem[O’Donoghue and Cand{\`e}s(2015)]{o2015adaptive}
Brendan O’Donoghue and Emmanuel Cand{\`e}s.
\newblock Adaptive restart for accelerated gradient schemes.
\newblock \emph{Foundations of computational mathematics}, 15\penalty0 (3):\penalty0 715--732, 2015.

\bibitem[O’{D}onoghue et~al.(2017)O’{D}onoghue, Chu, Parikh, and Boyd]{scs}
Brendan O’{D}onoghue, Eric Chu, Neal Parikh, and Stephen Boyd.
\newblock {SCS}: Splitting conic solver, version 3.27.
\newblock \url{https://github.com/cvxgrp/scs}, November 2017.

\bibitem[Parikh and Boyd(2014)]{parikh2014proximal}
Neal Parikh and Stephen Boyd.
\newblock Proximal algorithms.
\newblock \emph{Foundations and Trends{\textregistered} in Optimization}, 1\penalty0 (3):\penalty0 127--239, 2014.

\bibitem[Perron and Furnon(2024)]{ortools}
Laurent Perron and Vincent Furnon.
\newblock {OR-Tools}, 2024.
\newblock URL \url{https://developers.google.com/optimization/}.

\bibitem[Pock and Chambolle(2011)]{pock2011diagonal}
Thomas Pock and Antonin Chambolle.
\newblock Diagonal preconditioning for first order primal-dual algorithms in convex optimization.
\newblock In \emph{2011 International Conference on Computer Vision}, pages 1762--1769. IEEE, 2011.

\bibitem[Pock et~al.(2009)Pock, Cremers, Bischof, and Chambolle]{pock2009algorithm}
Thomas Pock, Daniel Cremers, Horst Bischof, and Antonin Chambolle.
\newblock An algorithm for minimizing the {M}umford-{S}hah functional.
\newblock In \emph{2009 IEEE 12th International Conference on Computer Vision}, pages 1133--1140. IEEE, 2009.

\bibitem[Renegar(1988)]{renegar1988polynomial}
James Renegar.
\newblock A polynomial-time algorithm, based on {N}ewton's method, for linear programming.
\newblock \emph{Mathematical Programming}, 40\penalty0 (1):\penalty0 59--93, 1988.

\bibitem[Renegar(2019)]{Renegar2019}
James Renegar.
\newblock Accelerated first-order methods for hyperbolic programming.
\newblock \emph{Mathematical Programming}, 173\penalty0 (1):\penalty0 1--35, 2019.

\bibitem[Rockafellar(1976)]{rockafellar1976monotone}
R~Tyrrell Rockafellar.
\newblock Monotone operators and the proximal point algorithm.
\newblock \emph{{SIAM} Journal on Control and Optimization}, 14\penalty0 (5):\penalty0 877--898, 1976.

\bibitem[Ruiz(2001)]{ruiz2001scaling}
Daniel Ruiz.
\newblock A scaling algorithm to equilibrate both rows and columns norms in matrices.
\newblock Technical report, CM-P00040415, 2001.

\bibitem[Ryu and Boyd(2016)]{ryu2016primer}
Ernest~K Ryu and Stephen Boyd.
\newblock Primer on monotone operator methods.
\newblock \emph{Appl. Comput. Math}, 15\penalty0 (1):\penalty0 3--43, 2016.

\bibitem[Stellato et~al.(2020)Stellato, Banjac, Goulart, Bemporad, and Boyd]{stellato2020osqp}
Bartolomeo Stellato, Goran Banjac, Paul Goulart, Alberto Bemporad, and Stephen Boyd.
\newblock {OSQP}: An operator splitting solver for quadratic programs.
\newblock \emph{Mathematical Programming Computation}, 12\penalty0 (4):\penalty0 637--672, 2020.

\bibitem[Vanderbei(2020)]{Vanderbei2020}
Robert~J. Vanderbei.
\newblock \emph{Problems in General Form}.
\newblock Springer International Publishing, Cham, Switzerland, 2020.

\bibitem[Wang and Shroff(2017)]{wang2017admmlp}
Sinong Wang and Ness Shroff.
\newblock A new alternating direction method for linear programming.
\newblock In \emph{Advances in Neural Information Processing Systems}, volume~30, 2017.

\bibitem[Wolsey(2021)]{wolsey2021integer}
Laurence~A Wolsey.
\newblock \emph{Integer programming}.
\newblock John Wiley \& Sons, Hoboken, NJ, USA, 2021.

\bibitem[Wright(2005)]{wright2005interior}
Margaret Wright.
\newblock The interior-point revolution in optimization: history, recent developments, and lasting consequences.
\newblock \emph{Bulletin of the American Mathematical Society}, 42\penalty0 (1):\penalty0 39--56, 2005.

\bibitem[Xiong and Freund(2024)]{xiong2024role}
Zikai Xiong and Robert~M Freund.
\newblock The role of level-set geometry on the performance of {PDHG} for conic linear optimization.
\newblock \emph{arXiv:2406.01942}, 2024.

\bibitem[Xpress(2014)]{xpress2014fico}
FICO Xpress.
\newblock {FICO Xpress Optimization Suite}.
\newblock 2014.

\bibitem[Yang and Lin(2018)]{yang2018rsg}
Tianbao Yang and Qihang Lin.
\newblock {RSG}: Beating subgradient method without smoothness and strong convexity.
\newblock \emph{Journal of Machine Learning Research}, 19\penalty0 (1):\penalty0 236--268, 2018.

\bibitem[Ye et~al.(1994)Ye, Todd, and Mizuno]{ye1994nl}
Yinyu Ye, Michael~J Todd, and Shinji Mizuno.
\newblock An ${O}(\sqrt{nL})$-iteration homogeneous and self-dual linear programming algorithm.
\newblock \emph{Mathematics of Operations Research}, 19\penalty0 (1):\penalty0 53--67, 1994.

\bibitem[Zanetti and Gondzio(2023)]{zanetti2023new}
Filippo Zanetti and Jacek Gondzio.
\newblock A new stopping criterion for {K}rylov solvers applied in interior point methods.
\newblock \emph{SIAM Journal on Scientific Computing}, 45\penalty0 (2):\penalty0 A703--A728, 2023.

\bibitem[Zhu and Chan(2008)]{zhu2008efficient}
Mingqiang Zhu and Tony Chan.
\newblock An efficient primal-dual hybrid gradient algorithm for total variation image restoration.
\newblock \emph{UCLA CAM Report}, 34:\penalty0 8--34, 2008.

\bibitem[Zubizarreta(2012)]{zubizarreta2012using}
Jos{\'e}~R Zubizarreta.
\newblock Using mixed integer programming for matching in an observational study of kidney failure after surgery.
\newblock \emph{Journal of the American Statistical Association}, 107\penalty0 (500):\penalty0 1360--1371, 2012.

\end{thebibliography}

\appendix

\SetKwFunction{PrimalFeasibilityStep}{PrimalFeasibilityStep}
\SetKwFunction{DualFeasibilityStep}{DualFeasibilityStep}
\SetKwFunction{PrimalDualStep}{PrimalDualStep}

\SetKwProg{Fn}{function}{}{end}
\newcommand{\Xfeas}{X_F}
\newcommand{\Yfeas}{Y_F}
\newcommand{\Zfeas}{Z_F}
\newcommand{\xfeas}{x_F}
\newcommand{\yfeas}{y_F}
\newcommand{\zfeas}{z_F}

\newpage

\section{Reduced cost selection for PDLP}\label{sec:reduced-cost-selection}
As discussed in Section~\ref{sec:pdhg}, there are multiple choices for recovering 
the reduced costs when running PDLP. These choices impact convergence detection but not the sequence of iterates. 
A natural choice is to set 
\[ 
r \gets \proj_{\mathcal{R}}(c - A^\T y);
\] 
however, an alternative that we found to be useful is to set 
\[
r_i \gets \begin{cases}
	c_i - a_i^\T y & \text{if } c_i - a_i^\T y > 0 \text{ and } x_i - (\ell_v)_i  \le \abs{x_i}   \\
	c_i - a_i^\T y & \text{if } c_i - a_i^\T y < 0 \text{ and } (u_v)_i - x_i  \le \abs{x_i}   \\
	0 & \text{otherwise}
\end{cases}
\]
where $a_i$ is the $i$th column of $A$. This choice is more robust to artificially large bounds.
For example, a relatively small residual reduced cost when a variable bound is extremely large, e.g., $10^{12},$ could prevent
convergence through its effect on the dual objective value.
This more advanced choice is only used when we run PDLP \emph{without} feasibility polishing.

\section{Derivation of \Cref{eq:specialized-pdhg}}\label{sec:derive-pdhg-for-this-paper}

This section shows how \Cref{eq:specialized-pdhg} follows from
specializing Algorithm 1 of \citet{chambolle2016ergodic} to our saddle-point formulation
\[
\min_{x \in \cX} \max_{y \in \cY}~
\Lag(x,y) = c^\top x - y^\top A x - p(y; -\uc, -\lc).
\]
The standard saddle-point form of \cite[Equation (2)]{chambolle2016ergodic} with $f(x) = 0$ and finite-dimensional spaces is
\[
\min_{x \in \R^{n}} \max_{y \in \R^{m}} y^\top K x + g(x) - h^{*}(y),
\]
where $g$ and $h$ are proper, lower semicontinuous, convex functions, and $h^*$ is the convex conjugate of $h$. 
This saddle-point problem reduces to our formulation by setting
\[
g(x) := c^\top x + I_{\cX}(x), \qquad
h^{*}(y) := I_{\cY}(y) +  p(y; -\uc, -\lc), \qquad
K := -A,
\]
where for any set $S$ and vector $s$ we define the indicator function
\[
I_{S}(s) := \begin{cases}
 \infty & \text{if } s \not\in S, \\
 0 & \text{if } s \in S.
\end{cases}
\]
The PDHG algorithm, i.e., Algorithm 1 of \citet{chambolle2016ergodic} with the Euclidean norm, takes the form 
\begin{subequations}
\begin{flalign}
x^{k+1} &= \argmin_{x} g(x) + (y^k)^\top K x + \frac{1}{2 \tau} \| x - x^k \|_2^2, \label{pdhg-primal-step} \\
y^{k+1} &= \argmin_{y} h^{*}(y) - y^\top K (2 x^{k+1} - x^k) + \frac{1}{2 \sigma} \| y - y^k \|_2^2, \label{pdhg-dual-step}
\end{flalign}
\end{subequations}
where $\tau,\sigma>0$ satisfy $\tau\sigma\|K\|_2^2 < 1$.

\paragraph{Primal step.}
To see that \Cref{pdhg-primal-step} reduces to \Cref{eq:specialized-pdhg:primal-step}, observe that
\begin{align*}
x^{k+1} &= \argmin_{x} g(x) + (y^k)^\top K x + \frac{1}{2 \tau} \| x - x^k \|_2^2 = \argmin_{x} c^\top x + I_{\cX}(x) + (y^k)^\top (-A) x + \frac{1}{2 \tau} \| x - x^k \|_2^2 \\
&= \argmin_{x \in \cX} c^\top x - (y^k)^\top A x + \frac{1}{2 \tau} \| x - x^k \|_2^2 = \argmin_{x \in \cX} (c - A^\top y^k)^\top x + \frac{1}{2 \tau} \| x - x^k \|_2^2 \\
&= \proj_{\cX}(x^k - \tau (c - A^\top y^k)).
\end{align*}

\paragraph{Dual step.}
To see that \Cref{pdhg-dual-step} reduces to \Cref{eq:specialized-pdhg:dual-step}, observe that
\begin{flalign}
y^{k+1} &= \argmin_{y} h^{*}(y) - y^\top K (2 x^{k+1} - x^k) + \frac{1}{2 \sigma} \| y - y^k \|_2^2 \notag \\
&= \argmin_{y} I_{\cY}(y) + p(y; -\uc, -\lc) - y^\top (-A) (2 x^{k+1} - x^k) + \frac{1}{2 \sigma} \| y - y^k \|_2^2 \notag \\
&= \argmin_{y \in \cY} p(y; -\uc, -\lc) + y^\top A (2 x^{k+1} - x^k) + \frac{1}{2 \sigma} \| y - y^k \|_2^2 \notag \\
&= \argmin_{y \in \cY} p(y; -\uc, -\lc) + \frac{1}{2 \sigma} \| y - (y^k - \sigma A (2 x^{k+1} - x^k)) \|_2^2 \notag \\
&= \prox_{\sigma p(\cdot; -\uc, -\lc)}(y^k - \sigma A (2 x^{k+1} - x^k)). \label{eq:y-k+1-prox-formula}
\end{flalign}
The proximal operator of $\sigma p(\cdot; -\uc, -\lc)$ can be written as
\[
\prox_{\sigma p(\cdot; -\uc, -\lc)}(v) = v - \sigma \proj_{[-\uc, -\lc]}(\sigma^{-1} v),
\]
due to Moreau's decomposition, i.e., $v = \prox_{f}(v) + \prox_{f^*}(v)$ for $f(v) = \sigma p(v; -\uc, -\lc)$, and the fact that convex conjugate of $f$ is $f^*(v) = I_{[-\uc, -\lc]}(\sigma^{-1} w)$. Substituting this expression into \Cref{eq:y-k+1-prox-formula} yields
\begin{align*}
y^{k+1} &= y^k - \sigma A (2 x^{k+1} - x^k) - \sigma \proj_{[-\uc, -\lc]}(\sigma^{-1} (y^k - \sigma A (2 x^{k+1} - x^k))) \\
&= y^k - \sigma A (2 x^{k+1} - x^k) - \sigma \proj_{[-\uc, -\lc]}(\sigma^{-1} y^k - A (2 x^{k+1} - x^k)),
\end{align*}
which matches \Cref{eq:specialized-pdhg:dual-step}.

\section{Theoretical basis for feasibility polishing}\label{app:feas-polishing-basis}

For simplicity, we focus on PDHG with a fixed step size less than $1 / \| A \|_2$
and a fixed frequency restart scheme. 
The former choice is typical  
for analyzing PDHG \cite{chambolle2011first,chambolle2016ergodic}.

Also for simplicity, we assume the LP instance takes the form $\min c^\T x$ such that $A x = b, x \ge 0$ and focus on primal feasibility polishing (i.e., $c=\zeros$).
However, this analysis generalizes to other settings such as the LP format described in Section~\ref{sec:pdhg} and dual feasibility polishing.
In this simple case, PDHG reduces to 
\begin{subequations}\label{eq:PDHG}
\begin{flalign}
x^{t+1} &\gets \proj_{x \ge 0}\left( x^{t} - \eta (c - A^\T y^{t}) \right) \\
y^{t+1} &\gets y^{t} - \eta (b - A (2 x^{t+1} - x^{t})) \ .
\end{flalign}
\end{subequations}
Let $\Xfeas := \{ x \ge 0 :  A x = b \}$ be the set of feasible solutions.
Let $\dist(X, x)  := \inf_{z \in X} \| x - z \|_2$.
The Hoffman constant  of $A$ is the largest constant, $\hoff{A}$, such that 
\begin{flalign}\label{eq:hoffman-bound}
\hoff{A} \dist( \Xfeas, x ) \le \| A x - b \|_2 \quad \quad \text{ for all } b \text{ such that } \Xfeas \neq \emptyset.
\end{flalign}
Thanks to \citet{hoffman_1952}, we know that this constant is always positive.
    
\begin{lemma}\label{lems-for-doing-restarts}
Suppose that $c = \zeros$ (so we have a primal feasibility problem), $0 < \eta < 1/\| A \|_2$ and $y^{0} = \zeros$.
Then PDHG, i.e., \eqref{eq:PDHG} with $\bx^{t} = \frac{1}{t} \sum_{i=1}^t x^{i}$ satisfies
\begin{flalign}
\| \bar{x}^{t} - x^{0} \|_2 &\le (q + 2) \dist(x^{0}, \Xfeas) \label{eq:dist-bound}  \\
\hoff{A} \dist( \Xfeas, \bx^{t} ) &\le \frac{ (q + 2) \dist(x^{0}, \Xfeas) }{\eta t}   \label{potential-reduction-guarrantee}
\end{flalign}
where $q = 4 \frac{1 + \eta \| A \|_2}{1 - \eta \| A \|_2}$.
\end{lemma}
    
\begin{proof}
Let $\ZStar$ denote the set of minimax solutions to $\min_{x  \ge 0} \max_{y}  b^\T y -y^\T A x$.
From Corollary 2 of \cite{applegate2023faster} we have 
\begin{flalign}\label{eq:upper-bound-z-bar-minus-z0}
\| \bar{z}^{t} - z^0 \|_2 \le (q + 2) \dist(z^0, \ZStar) = (q + 2) \dist(x^0, X_F) 
\end{flalign}
which by $\| \bar{x}^{t} - x^{0} \|_2 \le \| \bar{z}^{t} - z^0 \|_2$ gives \eqref{eq:dist-bound}.

Let $\bar{z}^{t} = \frac{1}{t} \sum_{i=1}^t z^{i}$.
By Corollary 2 of \cite{applegate2023faster}  with $C = 1/\eta$ and the fact that $y^{0} = \zeros$ we have
\begin{flalign}\label{eq:rho-upper-bound}
\rho_{\| \bar{z}^{t} - z^{0} \|_2}(\bar{z}^{t}) \le \frac{2 \| \bar{z}^{t} - z^{0} \|_2}{t \eta}.
\end{flalign}
Moreover, by \citet{applegate2023faster} equation~(22), we have 
\begin{flalign}\label{eq:rho-lower-bound-primal-feas}
\rho_{\| \bar{z}^{t} - z^{0} \|_2}(\bar{z}^{t}) \ge \|A \bx^{t} - b \|_2 + \| (-c + A^\T \by^{t})^{+} \|_2 \ge \|A \bx^{t} - b \|_2
\end{flalign}
where $+$ is the positive part operator. 
Combining \eqref{eq:hoffman-bound}, \eqref{eq:upper-bound-z-bar-minus-z0}, 
\eqref{eq:rho-upper-bound}, and \eqref{eq:rho-lower-bound-primal-feas} yields \eqref{potential-reduction-guarrantee}.
\end{proof}

    \begin{algorithm}[H]
    \SetAlgoLined
    
     {\bf Input:} 
     A starting point $x^{0,0}$, a step-size $\eta \in (0,\infty)$, a length for the restart period $\ell \in \N$\;
    Initialize outer loop counter $n\gets 0$\;
    Initialize duals $y^{0,0} \gets \zeros$\;

     \For{$n = 0,1,2, \dots$}{
     \textbf{initialize the inner loop.} inner loop counter $t\gets 0$ \;
     \Repeat{$t \ge \ell$
    }{
    	Take a step of PDHG, i.e., calculate $(x^{n,t+1}, y^{n,t+1})$ from $(x^{n,t}, y^{n,t})$  via \eqref{eq:PDHG} \;
        $\bx^{n,t+1}\gets\frac{1}{t+1} \sum_{i=1}^{t+1} x^{n,i}$\label{line:average}\;  \label{line:output-is-average-of-iterates} 
        $t\gets t+1$\;
     }
      \textbf{restart the outer loop.} $x^{n+1,0}\gets \bx^{n,t+1}$, $y^{n+1,0} \gets \zeros$, $n\gets n+1$\;
     }
     \caption{A first-order method for quickly finding a nearby feasible solution}
     \label{al:restarted+algorithm}
    \end{algorithm}

    \begin{theorem}\label{thm:feasibility-polishing}
    Suppose that $c = \zeros$ (so we have a primal feasibility problem).
    Consider Algorithm~\ref{al:restarted+algorithm} with $0 < \eta < 1/\| A \|_2$ and $\ell \ge t^\star \defeq \Bigg\lceil \frac{2 (q + 2)}{\eta \hoff{A}} \Bigg\rceil$ where $q = 4 \frac{1 + \eta \| A \|_2}{1 - \eta \| A \|_2}$.
    Then, for all $n \ge 1$
    \begin{subequations}
    \begin{flalign}
    \dist(x^{n,0}, \Xfeas) &\le 2^{-n} \dist(x^{0,0}, \Xfeas) \label{eq:distance-contraction} \\
    \| x^{n,0} - x^{0,0} \|_2 &\le 2 (q + 2) \dist(x^{0,0}, \Xfeas). \label{eq:dist-bound-to-feas}
    \end{flalign}
    \end{subequations}
    \end{theorem}

\begin{proof}
We obtain \eqref{eq:distance-contraction} from \eqref{potential-reduction-guarrantee} and the assumption that $\ell \ge t^\star$.
Observe that 
\begin{equation}
\begin{aligned} 
\| x^{k,0} - x^{0,0} \|_2 &\le \| x^{k,0} - x^{k-1,0} \|_2 + \| x^{k-1,0} - x^{0,0} \|_2 \le (q+2) \dist(x^{k-1,0}, \Xfeas) + \| x^{k-1,0} - x^{0,0} \|_2 \\
&\le (q+2) 2^{1-k} \dist(x^{0,0}, \Xfeas) + \| x^{k-1,0} - x^{0,0} \|_2 
\end{aligned}
\label{eq:inequality-to-prove-induct}
\end{equation}
where the first inequality uses the triangle inequality, the second inequality uses \eqref{eq:dist-bound}, and the last inequality uses \eqref{eq:distance-contraction}.
Using \eqref{eq:inequality-to-prove-induct} will prove the following induction hypothesis:
\[
\| x^{n,0} - x^{0,0} \|_2 \le \left( \sum_{i=1}^{n} 2^{1-i} \right) (q + 2) \dist(x^{0,0}, \Xfeas).
\]  
In particular, \eqref{eq:inequality-to-prove-induct} with $k = 1$ shows the hypothesis for $n=1$ since $x^{k-1, 0} = x^{0,0}$. 
If the hypothesis holds for $n = k - 1$ then by \eqref{eq:inequality-to-prove-induct} it holds for $n = k$. Thus we have proved the induction
hypothesis.
Using the formula for the sum of an infinite geometric series we get $\sum_{i=1}^{n} 2^{1-i} \le \sum_{i=1}^{\infty} 2^{1-i} = 2$, thus the 
induction hypothesis implies \eqref{eq:dist-bound-to-feas}.
\end{proof}

\Cref{eq:distance-contraction} implies quick convergence of feasibility polishing to a feasible solution.
\Cref{eq:dist-bound-to-feas} implies
feasibility polishing will converge to a feasible solution that is approximately optimal, assuming the starting point $x^{0,0}
$ was approximately optimal. In particular, \Cref{eq:dist-bound-to-feas} implies,
\begin{flalign}\label{eq:gap-doesnt-move-too-much}
\abs{c^\T x^{n,0} - c^\T x^{0,0}} \le 2 (q + 2) \| c \|_2 \dist(x^{0,0}, \Xfeas) \le \frac{2 (q + 2) \| c \|_2  \| A x^{0,0} - b \|_2}{\hoff{A}}.
\end{flalign}
Moreover, taking $\eta = \alpha / \| A \|_2$ where $\alpha \in (0,1)$, as typical \cite{chambolle2011first,chambolle2016ergodic}, gives $q = 4 \frac{1 + \alpha}{1 - \alpha}$. 
For example, $\alpha = 1/2$ gives $q=12$. 
Thus, from \Cref{eq:gap-doesnt-move-too-much}, we conclude that if we start at a point $x^{0,0}$ with 
$c^\T x^{0,0} \approx c^\T \xStar$ and $\| A x^{0,0} - b \|_2 \approx 0$ then 
\[
\lim_{n \rightarrow \infty} A x^{n,0} = b \text{ and } \limsup_{n \rightarrow \infty} c^\top x^{n,0} - c^\top \xStar \le \abs{c^\top x^{0,0} - c^\top \xStar} + \frac{2 (q + 2) \| c \|_2  \| A x^{0,0} - b \|_2}{\hoff{A}} \approx 0.
\]

\section{Infeasibility termination criteria}\label{sec:infeasibility-termination}

Our infeasibility termination criteria closely follow those of SCS \cite{scs} and \citet{applegate2024infeasibility}.

We declare the problem primal infeasible if we find a point $(y,r) \in \cY \times \cR$ satisfying:
\[
- p(-y; \lc, \uc) - p(-r ; \lv, \uv)  > 0,  \quad \frac{\| A^\T y + r \|_{\infty}}{- p(-y; \lc, \uc) - p(-r ; \lv, \uv)} \le 10^{-9}.
\]
Recall that 
\[
p(y; \ell, u) := u^\T y^+ -  \ell^\T y^-,
\]
and therefore
\[
-p(-y; \lc, \uc) - p(-r ; \lv, \uv) = \lc^\T y^{+} - \uc^\T y^{-}  + \lv^\T r^{+} - \uv^\T r^{-}.
\]
We declare the problem dual infeasible if we find a point $(x,s) \in \bar{\mathcal{X}} \times \bar{\mathcal{S}}$ satisfying:
\[
c^\T x < 0, \quad \frac{\| A x + s \|_{\infty}}{\abs{c^\T x}} \le 10^{-9}, \quad s \in \bar{\mathcal{S}}, \quad x \in \bar{\mathcal{X}},
\]
where $\bar{\mathcal{X}} \subseteq \R^{n}$ and $\bar{\mathcal{S}} \subseteq \R^{m}$ are Cartesian products whose $i$th components are given by 
\begin{equation*}
\begin{aligned}
&\bar{\mathcal{S}}_i := \begin{cases}
\{ 0 \} & (\lc)_i, (\uc)_i \in \R,  \\
\R^{+} & (\lc)_i = -\infty, ~ (\uc)_i \in \R ,\\
\R^{-} & (\lc)_i \in \R, ~ (\uc)_i = \infty, \\
\R & \text{otherwise}; 
\end{cases}
\quad \text{ and } \qquad
\end{aligned}
\begin{aligned}
&\bar{\mathcal{X}}_i := \begin{cases}
\{ 0 \} & (\lv)_i, (\uv)_i \in \R, \\
\R^{+} & (\lv)_i \in \R, (\uv)_i = \infty, \\
\R^{-} & (\lv)_i = -\infty, (\uv)_i \in \R, \\
\R & \text{otherwise}.
\end{cases}
\end{aligned}
\end{equation*}

\textbf{Remarks:}
\begin{itemize} 
\item The threshold of $10^{-9}$ was selected to ensure that none of the large-scale test instances would incorrectly terminate with an infeasibility declaration. %
\item All instances were rescaled prior to solving (\Cref{sec:instances}),
which implicitly affects the infeasibility termination criteria in a manner analogous to how rescaling affects 
the optimality termination criteria; see \Cref{eq:rescaled-termination-criteria}.
\end{itemize}

\section{Detailed statistics for PDLP on the large-scale instances}\label{sec:detailed-PDLP-stats-large-scale}

\begin{table}[H]
\centering
\caption{Detailed statistics for PDLP on the large-scale instances. KKT passes refers to the total number of times that $A x$ and $A^\top y$ are computed. Residuals are measured in $\ell_\infty$-norm, i.e., absolute error (on the instances scaled as per \Cref{sec:instances}). The gap is the difference between the primal and dual objective. Rows with gray background indicate unsuccessful termination (i.e., either hitting the six day time limit or in the case of qap-tho-150, a numerical error as per Table~\ref{table:numerical-results}), in which case we report the statistics from the most recent recorded iteration. Memory is the maximum memory usage in GB according to SLURM. Note that finer grained information is available  at \url{http://www.oliverhinder.com/large-scale-lp-problems}.}
\label{tab:detailed_pdlp_results}
\begin{tabular}{@{}l r@{\quad} r@{\quad} r@{\quad} r@{\quad} r@{\quad} r@{\quad} r@{\quad} r@{}}
\toprule
& \multirow{2}{*}{\begin{tabular}{@{}c@{}}Time \\ (hrs)\end{tabular}} & \multirow{2}{*}{\begin{tabular}{@{}c@{}}KKT \\ passes\end{tabular}} & \multicolumn{3}{c}{Objective} & \multicolumn{2}{c}{Residuals} & \multirow{2}{*}{\begin{tabular}{@{}c@{}}Memory \\ (GB)\end{tabular}} \\
\cmidrule(lr){4-6} \cmidrule(lr){7-8}
Instance name & & & Primal & Dual & Gap & Primal & Dual & \\
\midrule
\multicolumn{9}{l}{\textbf{With polishing}} \\
\midrule
design-match & 9.34 & 1.5E+4 & 2.87E+6 & 2.87E+6 & 5.5E-2 & 4.0E-11 & 0 & 135 \\
heat-source-easy & 59.97 & 1.9E+6 & 5.16E+7 & 5.16E+7 & 9.1E-4 & 9.9E-9 & 9.3E-9 & 18 \\
\rowcolor{black!10} heat-source-hard & --- & 3.8E+6 & 4.67E+7 & 4.69E+7 & -1.6E+5 & 6.2E+0 & 9.7E+2 & 17 \\
mediterranean-shipping & 32.62 & 1.1E+5 & -1.07E+3 & -1.07E+3 & 6.9E+0 & 7.8E-9 & 0 & 81 \\
\rowcolor{black!10} production-inventory & --- & 1.6E+6 & 4.50E+4 & 4.50E+4 & 3.3E-2 & 2.7E-7 & 5.9E-2 & 29 \\
\rowcolor{black!10} qap-tho-150 & 0.14 & 2.4E+2 & 0 & -4.23E+35 & 4.2E+35 & 1.0E+0 & 0 & 96 \\
qap-wil-100 & 0.28 & 3.8E+3 & 2.21E+5 & 2.21E+5 & 1.1E+2 & 5.7E-9 & 0 & 22 \\
supply-chain & 18.90 & 6.1E+4 & 1.08E+8 & 1.08E+8 & 6.7E+3 & 5.6E-9 & 0 & 70 \\
tsp-gaia-100m & 21.06 & 7.8E+3 & 1.94E+9 & 1.94E+9 & 5.8E+4 & 7.6E-9 & 0 & 611 \\
tsp-gaia-10m & 2.99 & 7.9E+3 & 1.30E+8 & 1.30E+8 & 3.0E+3 & 6.6E-9 & 0 & 40 \\
world-shipping & 53.81 & 2.2E+5 & -1.18E+5 & -1.18E+5 & 3.7E+2 & 9.9E-9 & 0 & 91 \\
\midrule
\multicolumn{9}{l}{\textbf{Without polishing}} \\
\midrule
design-match & 67.95 & 1.1E+5 & 2.87E+6 & 2.87E+6 & -8.7E-6 & 9.8E-9 & 0 & 136 \\
\rowcolor{black!10} heat-source-easy & --- & 3.9E+6 & 5.16E+7 & 5.16E+7 & -2.3E-2 & 2.8E-6 & 4.7E-5 & 15 \\
\rowcolor{black!10} heat-source-hard & --- & 3.5E+6 & 4.70E+7 & 4.76E+7 & -6.6E+5 & 1.2E+1 & 7.2E+2 & 15 \\
\rowcolor{black!10} mediterranean-shipping & --- & 5.5E+5 & -1.07E+3 & -1.07E+3 & 6.7E+0 & 7.4E-6 & 0 & 72 \\
\rowcolor{black!10} production-inventory & --- & 1.3E+6 & 4.50E+4 & 4.50E+4 & 8.7E-2 & 2.9E-7 & 5.9E-2 & 27 \\
\rowcolor{black!10} qap-tho-150 & 0.14 & 2.4E+2 & 0 & -4.23E+35 & 4.2E+35 & 1.0E+0 & 0 & 97 \\
qap-wil-100 & 44.35 & 6.4E+5 & 2.21E+5 & 2.21E+5 & 1.6E-4 & 6.6E-9 & 1.0E-8 & 20 \\
\rowcolor{black!10} supply-chain & --- & 8.5E+5 & 1.08E+8 & 1.08E+8 & 5.9E+1 & 2.8E-3 & 1.9E-5 & 60 \\
\rowcolor{black!10} tsp-gaia-100m & --- & 5.6E+4 & 1.94E+9 & 1.94E+9 & 2.8E+1 & 1.8E-3 & 0 & 552 \\
\rowcolor{black!10} tsp-gaia-10m & --- & 3.9E+5 & 1.30E+8 & 1.30E+8 & 2.3E-1 & 3.6E-5 & 0 & 37 \\
\rowcolor{black!10} world-shipping & --- & 2.4E+5 & -1.18E+5 & -1.18E+5 & 7.4E+1 & 8.5E-2 & 0 & 80 \\
\bottomrule
\end{tabular}
\end{table}

\section{Additional numerical experiments}\label{sec:additional-numerical-experiements}

To supplement the experiments on large-scale linear programs in the body of the paper, we ran some additional experiments to verify the performance on the Mittelmann benchmark (\Cref{sec:mittelmann-bench}), a standard collection of slightly smaller instances, and on a set of infeasible instances (\Cref{sec:primal-infeasible-tests}).
For consistency, PDLP was set up exactly the same as it was for the large-scale experiments, including
keeping presolved disabled, rescaling the instances, and maintaining the same termination criteria.
For these experiments, we used Gurobi version 12.0.3 instead of 11.0.2 (the version used for the large-scale experiments in \Cref{sec:numerical-results}) because at the time of running these experiments, the Pitt computing cluster we use had updated its Gurobi version.

\subsection{Numerical results on Mittelmann's benchmark instances}\label{sec:mittelmann-bench}

We compared PDLP with and without feasibility polishing on the Mittelmann test set.
Our results are reported in \Cref{fig:numerical-results-Mittelmann}, which shows that feasibility polishing 
is never much slower than standard PDLP and can give orders of magnitude speed improvements, consistent
 with our large-scale numerical results in \Cref{sec:numerical-results}.
 
 We also tested the impact of reducing the feasibility tolerance (both for the primal and dual). The default feasibility tolerance is $10^{-8}$ as specified in \Cref{eq:termination-criteria}.  
 These results are reported in \Cref{fig:numerical-results-Mittelmann-different-feasibility-tolerances}, which shows that PDLP with feasibility polishing is insensitive to reductions in feasibility tolerance, while PDLP without feasibility polishing suffers significant performance degradation as the feasibility tolerance is reduced.
 Without polishing, one additional instance failed to solve at $10^{-10}$ feasibility tolerance, and a further eight additional instances failed at $10^{-12}$ feasibility tolerance. With polishing, one additional instance failed to solve at $10^{-10}$ feasibility tolerance, and a further four additional instances failed at $10^{-12}$ feasibility tolerance.
 
\Cref{table:mittelmann-results} reports the results for PDLP (with and without polishing), and Gurobi (barrier, primal simplex, and dual simplex) on the Mittelmann benchmark instances. From this table it is clear that, with the exception of three instances, Gurobi is almost always faster. However, it is important to keep in mind that this CPU implementation of PDLP studied in this paper is not intended for tackling this scale of problem and a GPU implementation is likely more suitable for such a comparison.

 \subsubsection{Instances incorrectly declared infeasible}\label{sec:instances-incorrectly-declared-infeasible}
 
 PDLP, both with and without polishing, fails to find an optimal solution in 14 of 49 instances. 
 Two of these failures occur because the instance was declared infeasible by PDLP (both with and without feasibility polishing):
 pds-100 and a2864. We found that dropping the infeasibility termination tolerance to $10^{-10}$ prevents pds-100 from being declared infeasible,
 allowing the problem to be solved in 1075.1 seconds with feasibility polishing and in 1727.3 seconds without feasibility polishing.
Decreasing the infeasibility termination tolerance for a2864 does not prevent it from being declared infeasible,
and detailed investigation indicates that the issue is likely a small numerical error either 
in the computation of the infeasibility criteria or in the rescaling scheme (See \url{https://github.com/google/or-tools/issues/4874}). Manually modifying the infeasibility criteria from
 $- p(-y; \lc, \uc) - p(-r ; \lv, \uv)  > 0$ to $- p(-y; \lc, \uc) - p(-r ; \lv, \uv)  > 10^{-14}$ fixes the issue,
allowing the problem to be solved in 116.4 seconds with feasibility polishing and in 351.9 seconds without feasibility polishing.

\begin{figure}
\center
\includegraphics[width=0.95\textwidth]{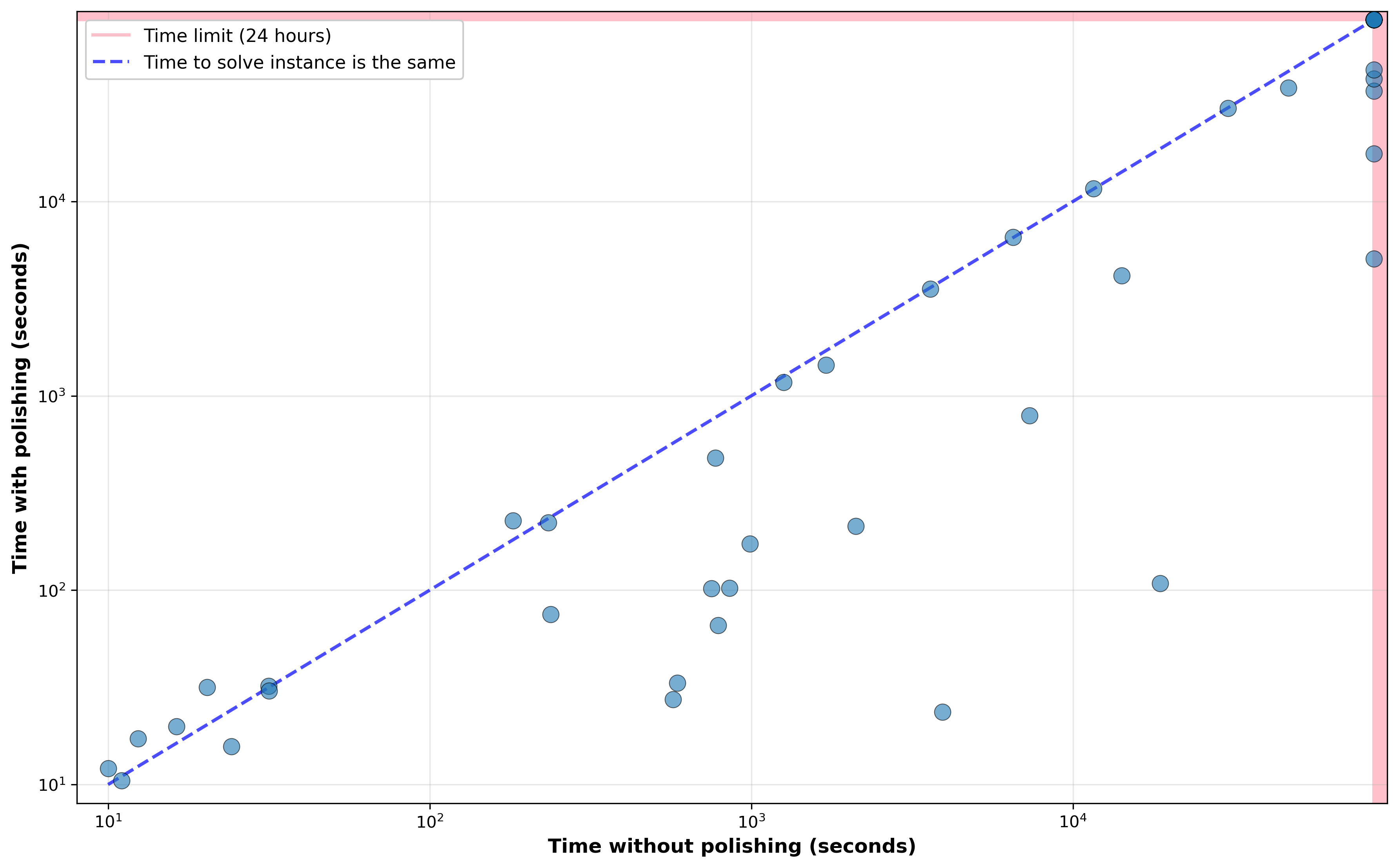}
\caption{Comparison of PDLP with and without feasibility polishing for the Mittelmann test set.
This test set consists of 49 problems, 14 of which were unsolved by both with and without polishing.
Note that all problems solved without polishing were also solved with the polishing option enabled.}
\label{fig:numerical-results-Mittelmann}
\end{figure}
\begin{figure}
\center
\includegraphics[width=0.95\textwidth]{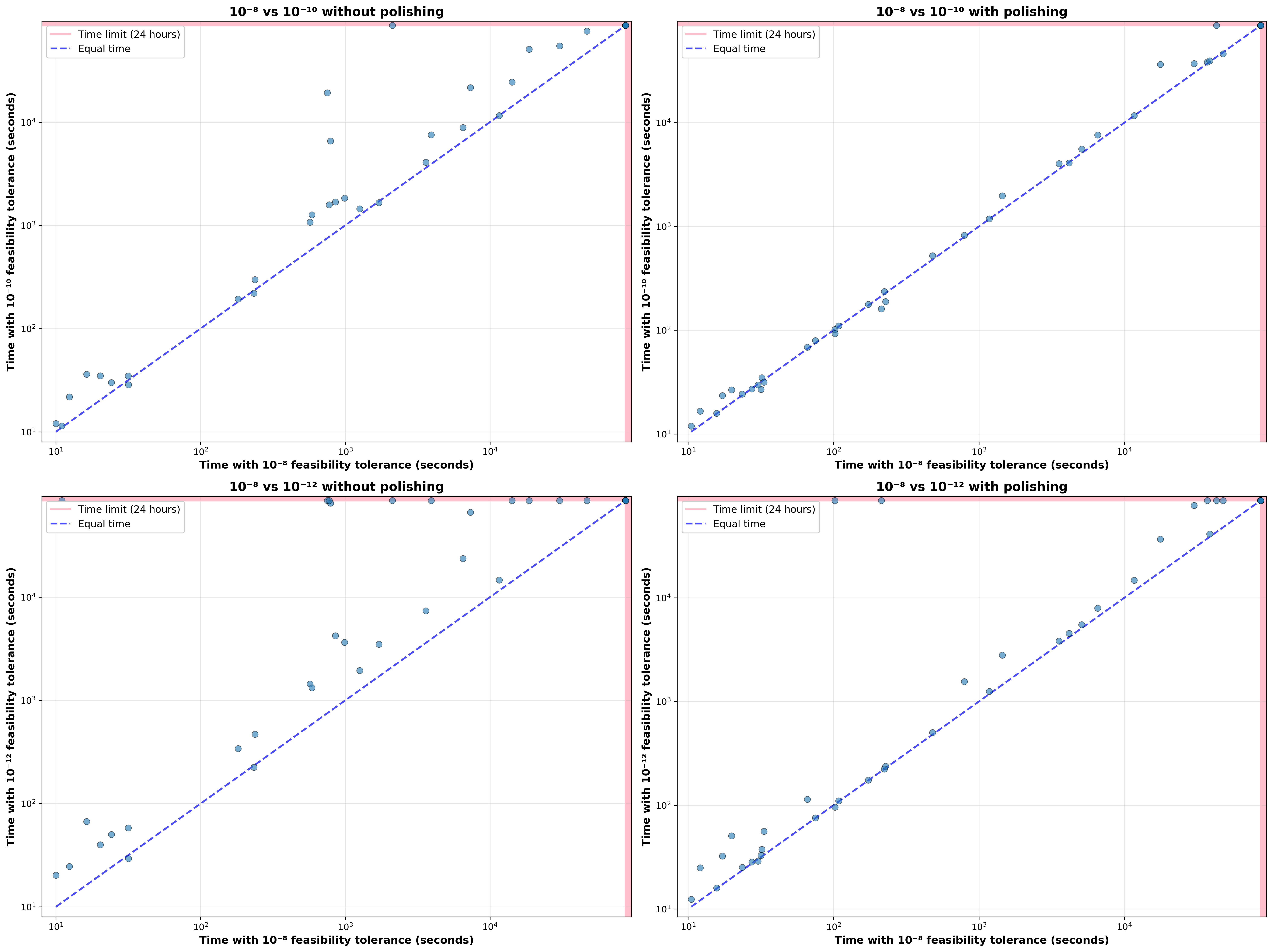}
\caption{Comparison of PDLP with and without feasibility polishing for the Mittelmann test set with different feasibility tolerances (both for the primal and dual).
The standard feasibility tolerance in this paper is $10^{-8}$.
}
\label{fig:numerical-results-Mittelmann-different-feasibility-tolerances}
\end{figure}

\begin{longtable}{LRRR}%
\caption{Statistics for Mittelmann's linear programming collection. Instances collected are based on
 the benchmark pages from 2025-06-18, which can be retrieved from archive.org:
 \url{https://web.archive.org/web/20250624111618/https://plato.asu.edu/ftp/lpfeas.html}.} \\
\toprule
\textbf{problem} & \textbf{\# constraints} & \textbf{\# variables} & \textbf{\# nonzeros} \\
\midrule 
\endfirsthead

\multicolumn{4}{c}{\tablename\ \thetable\ -- \textit{Continued from previous page}} \\
\toprule
\textbf{problem} & \textbf{\# constraints} & \textbf{\# variables} & \textbf{\# nonzeros} \\
\midrule
\endhead

\midrule
\multicolumn{4}{r}{\textit{Continued on next page}} \\
\endfoot

\bottomrule
\endlastfoot 
  
    L1\_sixm250obs & 986069 & 428032 & 4280320 \\
    Linf\_520c & 93326 & 69004 & 566193 \\
    a2864 & 22117 & 200787 & 20078717 \\
    bdry2 & 376500 & 250998 & 1500003 \\
    cont1 & 160793 & 40398 & 399991 \\
    cont11 & 160793 & 80396 & 439989 \\
    datt256 & 11077 & 262144 & 1503732 \\
    dlr1 & 1735470 & 9121907 & 18365107 \\
    ex10 & 69609 & 17680 & 1179680 \\
    fhnw-bin1 & 772872 & 1141653 & 8611326 \\
    fome13 & 48569 & 97840 & 334984 \\
    graph40-40 & 360900 & 102600 & 1260900 \\
    irish-e & 104260 & 61728 & 538809 \\
    neos & 479120 & 36786 & 1084461 \\
    neos3 & 512209 & 6624 & 1542816 \\
    neos-3025225 & 91572 & 69846 & 9357951 \\
    neos5052403 & 38269 & 32868 & 4898304 \\
    neos-5251015 & 486531 & 136971 & 1955388 \\
    ns1687037 & 50622 & 43749 & 1406739 \\
    ns1688926 & 32768 & 16587 & 1712128 \\
    nug08-3rd & 19728 & 20448 & 139008 \\
    pds-100 & 156244 & 505360 & 1390539 \\
    psched3-3 & 266228 & 79555 & 1062480 \\
    qap15 & 6331 & 22275 & 110700 \\
    rail02 & 95791 & 270869 & 756228 \\
    rail4284 & 4284 & 1092610 & 12372358 \\
    rmine15 & 358395 & 42438 & 879732 \\
    s82 & 87878 & 1690631 & 7022608 \\
    s100 & 14734 & 364417 & 2127672 \\
    s250r10 & 10963 & 273142 & 1572104 \\
    savsched1 & 295990 & 328575 & 1846351 \\
    scpm1 & 5000 & 500000 & 6250000 \\
    shs1023 & 133944 & 444625 & 1044725 \\
    square41 & 40161 & 62234 & 13628623 \\
    stat96v2 & 29089 & 957432 & 2852184 \\
    stormG2\_1000 & 528186 & 1259121 & 4228817 \\
    stp3d & 159488 & 204880 & 662128 \\
    support10 & 165685 & 14770 & 551152 \\
    tpl-tub-ws16 & 1154615 & 747691 & 4720567 \\
    woodlands09 & 194599 & 382147 & 2646003 \\
    Dual2\_5000 & 30000600 & 33050602 & 93001800 \\
    Primal2\_1000 & 1299380 & 2559380 & 5498140 \\
    thk\_48 & 6366377 & 8609262 & 27802878 \\
    thk\_63 & 5694387 & 7701112 & 21592414 \\
    L1\_sixm1000obs & 3082940 & 1426256 & 14262560 \\
    L2CTA3D & 210000 & 10000000 & 30000000 \\
    degme & 185501 & 659415 & 8127528 \\
    dlr2 & 7132926 & 38868107 & 78091589 \\
    set-cover & 10000 & 1102008 & 20442268 \\

\end{longtable}

\begin{longtable}{CRRRRRRR}
\caption{Computational results for Mittelmann benchmark collection.
The value --- signifies the problem failed to solve in the 24 hour time limit, $\dagger$ indicates the method ran out of memory, and $\cross$
signifies the problem was incorrectly declared infeasible (See discussion in Section~\ref{sec:instances-incorrectly-declared-infeasible}). One KKT pass refers to the calculation $A x$ and $A^\top y$.}\label{table:mittelmann-results} \\
\toprule
\textbf{Problem} & \multicolumn{4}{C}{\textbf{PDLP}} & \multicolumn{3}{C}{\textbf{Gurobi (minutes)}} \\
\cmidrule(lr){2-5} \cmidrule(lr){6-8}
& \multicolumn{2}{C}{\textbf{KKT Passes}} & \multicolumn{2}{C}{\textbf{Time (minutes)}} & \textbf{Barrier} & \multicolumn{2}{C}{\textbf{Simplex}} \\
\cmidrule(lr){2-3} \cmidrule(lr){4-5} \cmidrule(lr){7-8}
& \textbf{no polish} & \textbf{polish} & \textbf{no polish} & \textbf{polish} &  & \textbf{Primal} & \textbf{Dual} \\
\midrule
\endfirsthead

\multicolumn{8}{C}{\tablename\ \thetable\ -- \textit{Continued from previous page}} \\
\toprule
\textbf{Problem} & \multicolumn{4}{C}{\textbf{PDLP}} & \multicolumn{3}{C}{\textbf{Time for Gurobi}} \\
\cmidrule(lr){2-5} \cmidrule(lr){6-8}
& \multicolumn{2}{C}{\textbf{KKT Passes}} & \multicolumn{2}{C}{\textbf{Time}} & \textbf{Barrier} & \multicolumn{2}{C}{\textbf{Simplex}} \\
\cmidrule(lr){2-3} \cmidrule(lr){4-5} \cmidrule(lr){7-8}
& \textbf{no polish} & \textbf{polish} & \textbf{no polish} & \textbf{polish} &  & \textbf{Primal} & \textbf{Dual} \\
\midrule
\endhead

\midrule
\multicolumn{8}{R}{\textit{Continued on next page}} \\
\endfoot

\bottomrule
\endlastfoot

Dual2\_5000 & 79,497 & 79,497 & 193.36 & 194.59 & $\dagger$ & --- & \textbf{59.19} \\
L1\_sixm1000obs & 69,577 & 7,049 & 35.19 & 3.55 & 1.46 & 2.18 & \textbf{0.53} \\
L1\_sixm250obs & 53,324 & 6,733 & 12.53 & 1.70 & 0.11 & 0.11 & \textbf{0.09} \\
L2CTA3D & 19,872 & 987 & 13.16 & \textbf{1.10} & 22.03 & --- & --- \\
Linf\_520c & --- & --- & --- & --- & \textbf{0.09} & 9.20 & 1.00 \\
Primal2\_1000 & --- & 922,873 & --- & \textbf{293.87} & --- & 1181.92 & 1236.06 \\
a2864 & $\times$ & $\times$ & $\times$ & $\times$ & \textbf{0.03} & 5.06 & 6.20 \\
bdry2 & --- & --- & --- & --- & \textbf{0.27} & 71.18 & 96.51 \\
cont1 & --- & 3,331,570 & --- & 618.30 & \textbf{0.03} & 5.13 & 5.89 \\
cont11 & --- & 3,693,319 & --- & 714.41 & \textbf{0.04} & 38.09 & 25.66 \\
datt256 & 1,347 & 1,499 & 0.27 & 0.33 & \textbf{0.03} & 225.96 & 5.28 \\
degme & --- & --- & --- & --- & \textbf{0.73} & --- & --- \\
dlr1 & --- & --- & --- & --- & \textbf{0.91} & 210.87 & 39.98 \\
dlr2 & --- & --- & --- & --- & \textbf{9.14} & --- & --- \\
ex10 & 710 & 788 & 0.17 & 0.20 & --- & \textbf{0.06} & 0.08 \\
fhnw-bin1 & 357,434 & 357,594 & 108.48 & 109.10 & 2.54 & \textbf{0.45} & 3.86 \\
fome13 & --- & 440,673 & --- & 84.78 & \textbf{0.02} & 0.12 & 0.13 \\
graph40-40 & 1,090 & 1,286 & 0.21 & 0.29 & \textbf{0.02} & 0.39 & 0.12 \\
irish-e & --- & --- & --- & --- & \textbf{0.05} & 0.74 & 0.39 \\
neos & --- & 4,054,159 & --- & 794.93 & 0.15 & 1.87 & \textbf{0.11} \\
neos-3025225 & 11,141 & 11,141 & 3.02 & 3.80 & \textbf{0.21} & 7.83 & 2.55 \\
neos-5251015 & 2,805 & 2,805 & 0.53 & 0.54 & 0.06 & 47.87 & \textbf{0.05} \\
neos3 & 854 & 822 & 0.18 & 0.17 & \textbf{0.03} & 0.43 & 0.05 \\
neos5052403 & 40,302 & 1,700 & 9.51 & 0.46 & \textbf{0.08} & 2.83 & 0.85 \\
ns1687037 & --- & --- & --- & --- & \textbf{0.08} & 10.56 & 4.58 \\
ns1688926 & --- & --- & --- & --- & 0.03 & 1.16 & \textbf{0.02} \\
nug08-3rd & 1,537 & 1,689 & 0.34 & 0.53 & \textbf{0.01} & 0.82 & 0.62 \\
pds-100 & $\times$ & $\times$ & $\times$ & $\times$ & 0.25 & 1.97 & \textbf{0.17} \\
psched3-3 & --- & --- & --- & --- & \textbf{0.13} & 1.51 & 1.74 \\
qap15 & 50,434 & 1,926 & 9.81 & 0.55 & \textbf{0.01} & 0.21 & 0.30 \\
rail02 & --- & --- & --- & --- & \textbf{0.69} & 52.74 & 39.41 \\
rail4284 & 219,460 & 960 & 65.48 & 0.39 & \textbf{0.32} & 30.54 & 0.80 \\
rmine15 & 70,824 & 7,834 & 14.26 & 1.70 & \textbf{0.25} & 10.29 & 8.79 \\
s100 & 117,222 & 105,378 & 20.99 & 19.60 & \textbf{0.12} & 1.33 & 1.71 \\
s250r10 & 2,812,856 & 2,864,646 & 505.42 & 502.39 & \textbf{0.11} & 1.41 & 0.32 \\
s82 & 2,082,799 & 1,808,861 & 780.22 & 642.19 & \textbf{1.26} & 165.40 & 20.05 \\
savsched1 & 1,955 & 966 & 0.40 & 0.26 & \textbf{0.13} & 0.46 & 1.74 \\
scpm1 & 1,815 & 1,639 & 0.53 & 0.51 & \textbf{0.25} & 13.33 & 4.54 \\
set-cover & 688,191 & 3,385 & 310.96 & 1.80 & \textbf{1.69} & 755.10 & 1041.93 \\
shs1023 & --- & --- & --- & --- & \textbf{0.29} & 17.77 & 1.12 \\
square41 & 242,433 & 232,357 & 60.14 & 59.14 & \textbf{0.03} & 0.06 & 0.05 \\
stat96v2 & --- & --- & --- & --- & \textbf{0.38} & 108.24 & 30.46 \\
stormG2\_1000 & 48,713 & 30,365 & 12.90 & 7.98 & 0.58 & 13.71 & \textbf{0.38} \\
stp3d & 92,890 & 15,187 & 16.48 & 2.89 & \textbf{0.07} & 10.50 & 1.39 \\
support10 & 21,540 & 6,736 & 3.96 & 1.25 & \textbf{0.06} & 0.69 & 0.28 \\
thk\_48 & 140,512 & 14,658 & 122.21 & \textbf{13.20} & 58.65 & --- & --- \\
thk\_63 & 32,774 & 29,894 & 28.45 & 24.06 & \textbf{2.18} & --- & 12.94 \\
tpl-tub-ws16 & 713,992 & 208,841 & 236.37 & 69.26 & \textbf{0.54} & 6.00 & 13.90 \\
woodlands09 & 17,870 & 17,870 & 3.89 & 3.72 & \textbf{0.08} & 4.77 & 8.32 \\

\end{longtable}

\subsection{Numerical results for primal infeasibility detection}\label{sec:primal-infeasible-tests}

We compared PDLP with and without feasibility polishing on the NETLIB infeasibility test set (See \url{https://www.netlib.org/lp/infeas/readme}).
This test set mostly comprises smaller instances (See \Cref{table:infeasibility-problem-stats}), but it can be a good test of robustness. 

PDLP performance on this set is reported in \Cref{table:infeasibility-performance}. 
Overall, PDLP's performance on these problems is less than desired, with it only solving 15/29 problems.
Thus, improving the ability of first-order methods to reliably detect infeasibility is an important subject for future research.
We also find that on these instances, feasibility polishing has no notable impact on solver performance.

We also tested Gurobi on these problems but do not report the results because it was so dominant. In particular, every instance was declared primal infeasible by at least one Gurobi method in less than 0.1 seconds, except `pang', which was declared either infeasible or unbounded by both dual and primal simplex in 0.01 seconds.

\begin{longtable}{LRRR} %
\caption{Problem statistics for primal infeasible linear programs.}\label{table:infeasibility-problem-stats} \\
\toprule
\textbf{problem} & \textbf{\# constraints} & \textbf{\# variables} & \textbf{\# nonzeros} \\
\midrule 
\endfirsthead

\multicolumn{4}{C}{\tablename\ \thetable\ -- \textit{Continued from previous page}} \\
\toprule
\textbf{problem} & \textbf{\# constraints} & \textbf{\# variables} & \textbf{\# nonzeros} \\
\midrule
\endhead

\midrule
\multicolumn{4}{R}{\textit{Continued on next page}} \\
\endfoot

\bottomrule
\endlastfoot

bgdbg1 & 348 & 407 & 1440 \\
bgetam & 400 & 688 & 2409 \\
bgindy & 2671 & 10116 & 65502 \\
bgprtr & 20 & 34 & 64 \\
box1 & 231 & 261 & 651 \\
ceria3d & 3576 & 824 & 17602 \\
chemcom & 288 & 720 & 1566 \\
cplex1 & 3005 & 3221 & 8944 \\
cplex2 & 224 & 221 & 1058 \\
ex72a & 197 & 215 & 467 \\
ex73a & 193 & 211 & 457 \\
forest6 & 66 & 95 & 210 \\
galenet & 8 & 8 & 16 \\
gosh & 3792 & 10733 & 97231 \\
gran & 2658 & 2520 & 20106 \\
greenbea & 2393 & 5405 & 30883 \\
itest2 & 9 & 4 & 17 \\
itest6 & 11 & 8 & 20 \\
klein1 & 54 & 54 & 696 \\
klein2 & 477 & 54 & 4585 \\
klein3 & 994 & 88 & 12107 \\
mondou2 & 312 & 604 & 1208 \\
pang & 361 & 460 & 2652 \\
pilot4i & 410 & 1000 & 5141 \\
qual & 323 & 464 & 1646 \\
reactor & 318 & 637 & 2420 \\
refinery & 323 & 464 & 1626 \\
vol1 & 323 & 464 & 1646 \\
woodinfe & 35 & 89 & 140 \\

\end{longtable}

\begin{longtable}{RRRRR}
\caption{Computational results for the collection of primal infeasible linear programs given in \Cref{table:infeasibility-problem-stats}. The time limit for these problems was one hour.
One KKT pass refers to the calculation $A x$ and $A^\top y$.}\label{table:infeasibility-performance} \\
\toprule
\textbf{Problem} & \multicolumn{2}{C}{\textbf{Time (seconds)}} & \multicolumn{2}{C}{\textbf{KKT Passes}} \\
\cmidrule(lr){2-3} \cmidrule(lr){4-5}
& \textbf{no polishing} & \textbf{polishing} & \textbf{no polishing} & \textbf{polishing} \\
\midrule
\endfirsthead

\multicolumn{5}{C}{\tablename\ \thetable\ -- \textit{Continued from previous page}} \\
\toprule
\textbf{Problem} & \multicolumn{2}{C}{\textbf{Time (seconds)}} & \multicolumn{2}{C}{\textbf{KKT Passes}} \\
\cmidrule(lr){2-3} \cmidrule(lr){4-5}
& \textbf{no polishing} & \textbf{polishing} & \textbf{no polishing} & \textbf{polishing} \\
\midrule
\endhead

\midrule
\multicolumn{5}{r}{\textit{Continued on next page}} \\
\endfoot

\bottomrule
\endlastfoot

bgdbg1 & 104.0 & 87.6 & 7379 & 7379 \\
bgetam & 239.9 & 283.7 & 22160 & 22160 \\
bgindy & 1.9 & 2.0 & 68 & 68 \\
bgprtr & 14.3 & 15.6 & 1857 & 1857 \\
box1 & 139.0 & 142.0 & 11973 & 11973 \\
ceria3d & --- & --- & --- & --- \\
chemcom & 34.7 & 36.2 & 3081 & 3081 \\
cplex1 & --- & --- & --- & --- \\
cplex2 & --- & --- & --- & --- \\
ex72a & 360.6 & 380.7 & 31990 & 31990 \\
ex73a & 438.1 & 554.3 & 41611 & 41623 \\
forest6 & 64.3 & 66.6 & 5570 & 5570 \\
galenet & 0.1 & 0.1 & 67 & 67 \\
gosh & --- & --- & --- & --- \\
gran & --- & --- & --- & --- \\
greenbea & 231.4 & 214.5 & 18131 & 18131 \\
itest2 & 0.1 & 0.1 & 206 & 206 \\
itest6 & 0.1 & 0.1 & 66 & 66 \\
klein1 & --- & --- & --- & --- \\
klein2 & --- & --- & --- & --- \\
klein3 & --- & --- & --- & --- \\
mondou2 & 113.7 & 114.2 & 7693 & 7693 \\
pang & --- & --- & --- & --- \\
pilot4i & --- & --- & --- & --- \\
qual & --- & --- & --- & --- \\
reactor & --- & --- & --- & --- \\
refinery & --- & --- & --- & --- \\
vol1 & --- & --- & --- & --- \\
woodinfe & 41.4 & 37.8 & 3969 & 3969 \\

\end{longtable}

\end{document}